\documentclass{daj}

\dajAUTHORdetails{%
  title = {Diophantine Equations in Semiprimes}, %% please capitalize all significant words
  author = {Shuntaro Yamagishi},
  plaintextauthor = {Shuntaro Yamagishi},
  keywords = {Hardy-Littlewood circle method, Diophantine equations, almost primes},
}   %%% END \dajAUTHORdetails

\usepackage{bbm}
\usepackage{graphicx}
\usepackage{hyperref}
\usepackage{setspace}
\usepackage{enumerate}
\usepackage{amsmath,amsfonts,amssymb,amscd, amsthm}
\usepackage{bbm}

\newtheorem{thm}{Theorem}[section]
\newtheorem{cor}[thm]{Corollary}
\newtheorem{lem}[thm]{Lemma}
\newtheorem{prop}[thm]{Proposition}

\dajEDITORdetails{%
   year={2019},
   number={17},
   received={5 September 2018},   % received date: example: 7 January 2017
   published={21 November 2019},  % published date
   doi={10.19086/da.11075},       % XXX = number of paper, e.g. da006 for paper#6
}   %%% END \dajEDITORdetails

\begin{document}

\begin{frontmatter}[classification=text]
%% EDITOR: this will force the keywords to appear right after the Abstract.
%%   If the abstract is too long and would force the keywords off the
%%   front page, please comment out % [classification=text] above
%%   This way the keywords will be floated on the bottom of the first page
%%   even though the Abstract spills over to the next page.

%%% AUTHOR: Title goes here.  This line is optional.  You must use it
%%   if title has footnote attached or requires nontrivial typesetting,
%%   e.g., inclusion of linebreaks to force nice layout.
\title{Diophantine Equations in Semiprimes} % \footnote{This is a footnote to the title}} %% please capitalize all significant words

%%% AUTHOR:
%%% List all authors. If you wish, place grant acknowledgements in \thanks.
%%% In brackets include a short tag for each author.
\author[sy]{Shuntaro Yamagishi\thanks{Supported by the NWO Veni Grant \texttt{016.Veni.192.047}}}
%\author[johan]{Johan H{\aa}stad\thanks{Supported by...}}
%\author[laci]{L\'aszl\'o Lov\'asz\thanks{Supported by...}}
%\author[andy]{Andrew Chi-Chih Yao\thanks{Supported by...}}

%%% AUTHOR: Abstract goes here
\begin{abstract}
A semiprime is a natural number which is the product of two (not necessarily distinct) prime numbers.
Let $F(x_1, \ldots, x_n)$ be a degree $d$ homogeneous form with integer coefficients.
We provide sufficient conditions, similar to those of the seminal work of B. J. Birch \cite{B}, for which the equation
$F (x_1, \ldots, x_n) = 0$ has infinitely many integer solutions with semiprime coordinates.
Previously it was known, by a result of {\'A}. Magyar and T. Titichetrakun \cite{MT}, that under the same hypotheses
there exist infinitely many integer solutions to the equation with coordinates that have at most $384 n^{3/2} d (d+1)$ prime factors.
\end{abstract}
\end{frontmatter}

%%% AUTHOR: body of paper starts here
\section{Introduction}
\label{sec intro}

Solving Diophantine equations in primes or almost primes is a fundamental problem in number theory.
%By almost primes we mean numbers with a bounded number of prime factors.
For example, the celebrated work of B. Green and
T. Tao \cite{GT0} on arithmetic progressions in primes can be phrased as the statement that given any $n \in \mathbb{N}$ the system of linear equations
%on arithmetic progressions in primes can be stated as the system of linear equations
$
x_{i+2} - {x_{i+1}} = x_{i+1} - x_{i} \ (1 \leq i \leq n)
$
has a solution $(p_1, \ldots, p_{n+2})$ such that each $p_i$ is prime %for each $1 \leq i \leq n+2$
and $p_1< p_2 < \ldots < p_{n+2}$.
A major achievement extending this result in which a more general system of linear equations is considered has been established by B. Green, T. Tao, and T. Ziegler (see \cite{GT1}, \cite{GT2}, \cite{GTZ}) and we refer the reader to \cite[Theorem 1.8]{GT1} for the precise statement.
Another important achievement in this area is the well-known theorem of Chen \cite{C} related to the twin prime conjecture.
The theorem asserts that the equation
$
x_1 - x_2 = 2
$
has infinitely many solutions $(\ell_1, p_2)$ where $\ell_1$ has at most two prime factors and $p_2$ is prime.

The main focus of this paper is on equations involving higher degree polynomials. Let $d > 1$.
Let $F(\mathbf{x})$ be a degree $d$ homogeneous form in $\mathbb{Z}[x_1, \ldots, x_n]$. We are interested in integer solutions $\mathbf x$ to the
equation
\begin{equation}
\label{eqn main}
F(\mathbf{x}) = 0
\end{equation}
for which all coordinates have small numbers of prime factors. For this to be possible one has to impose appropriate conditions. 
Let $\mathbb{Z}_p^{\times}$ be the units of $p$-adic integers. We consider the following conditions.  %Then these conditions are as follows.
\medskip

%We say that the equation (\ref{eqn main}) satisfies the \emph{local conditions} ($\star$)
\noindent \textbf{Local conditions ($\star$).} The equation (\ref{eqn main}) has a non-singular real solution in $(0,1)^{n}$,
and also %the system of equations (\ref{eqn main})
has a non-singular solution in $( \mathbb{Z}_p^{\times})^n$ for every prime $p$.
\medskip

%We also let $V^*_{F} \subseteq \mathbb{A}_{\mathbb{C}}^{n}$ be the \textit{Birch singular locus} of ${F}$, which is an affine variety defined by
Let $V^*_{F}$ be an affine variety in $\mathbb{A}^n_{\mathbb{C}}$ defined by
\begin{equation}
\label{sing loc}
V_{F}^* := \left\{ \mathbf{z} \in \mathbb{C}^n:  %\mathbb{A}_{\mathbb{C}}^{n} :
\frac{\partial F}{\partial x_j} (\mathbf{z}) = 0  \  ( 1 \leq j \leq n)  \right\}.
\end{equation}
By Euler's formula it follows that $V_F^*$ is the singular locus of $V(F) = \{ \mathbf{z} \in \mathbb{C}^n : F(\mathbf{z}) = 0 \}$,
%(this is not necessarily the case for Birch singular loci in more general settings),
but we shall consider it as a subvariety of $\mathbb{A}_{\mathbb{C}}^{n}$ and let $\textnormal{codim } V_{F}^* = n - \dim V_{F}^*$.

For solving general non-linear polynomial equations in primes, the following important result was established by B. Cook and \'{A}. Magyar \cite{CM}.
\begin{thm} \cite[Theorem 1]{CM}
\label{thm CM}
Let $F(\mathbf{x}) \in \mathbb{Z}[x_1, \ldots, x_n]$ be a degree $d$ homogeneous form. Suppose that $F$ satisfies the local conditions \textnormal{($\star$)} and $\textnormal{codim } V_{F}^*$ is sufficiently large with respect to $d$. Then the equation (\ref{eqn main}) has an infinite number of solutions $(p_1, \ldots, p_n)$
for which $p_i$ is prime for each $1 \leq i \leq n$.
\end{thm}
Here the theorem requires $\textnormal{codim } V_{F}^*$ to be very large. In fact, the required bound on $\textnormal{codim } V_{F}^*$ ``already exhibit(s) tower type behavior in $d$''\cite{CM}. We also refer the reader to \cite{Zh} for the case of quadratic forms.
It is expected that a lower bound exponential in $d$ is sufficient in Theorem \ref{thm CM} \cite{CM},
because this is the case for integer solutions as seen in the work of B. J. Birch \cite{B}.
As the requirement on $\textnormal{codim } V_{F}^*$ in Theorem \ref{thm CM} is significantly larger than what is expected, it is natural to consider if one can
%solve (\ref{eqn main}) in
achieve a result analogous to Theorem \ref{thm CM} for almost primes, which are positive integers with a small number of prime factors (counting multiplicity), with smaller  $\textnormal{codim } V_{F}^*$. In this direction, there is a result by {\'A}. Magyar and T. Titichetrakun \cite{MT}
provided $\textnormal{codim } V_{F}^* > 2^{d} (d-1)$, which is also the required bound in \cite{B}.
%Under the same assumption on $\textnormal{codim } V_{F}^*$ as in \cite{B}, namely $\textnormal{codim } V_{F}^* > 2^{d} (d-1)$,
%the upper bound on the number of prime factors $384 n^{3/2} d (d+1)$ in \cite[Theorem 1.1]{MT} is reduced to $c'(d)  n$, where $c'(d) > 0$ is a constant depending only on $d$, in \cite[Theorem 9.1]{SS}. This bound is improved even further when $\textnormal{codim } V_{F}^* > 2^{d+1}(d^2-1)$ where D. Schindler and E. Sofos establish the following.
\begin{thm} \cite[Theorem 1.1]{MT}
\label{thm SS}
Let $F(\mathbf{x}) \in \mathbb{Z}[x_1, \ldots, x_n]$ be a degree $d$ homogeneous form.
Suppose that $F$ satisfies the local conditions \textnormal{($\star$)} and
$\textnormal{codim } V_{F}^* > 2^{d}(d-1).$
Then the equation (\ref{eqn main}) has an infinite number of solutions $(\ell_1, \ldots, \ell_n)$
such that $\ell_i$ has at most $384 n^{3/2} d (d+1)$ prime factors for each $1 \leq i \leq n$.
\end{thm}
This result was established by combining sieve methods with the Hardy-Littlewood circle method.
In order to keep the amount of notation to a minimum we presented simplified statements of Theorems \ref{thm CM} and \ref{thm SS} (without quantitative estimates and only the case of one homogeneous form instead of systems of homogeneous forms of equal degree); we refer the reader to the respective papers for the precise statements. We also refer the reader to \cite[Section 5.2]{KT} and \cite[Section 17]{VW} for overviews of the progress on a related problem, the Goldbach-Waring problem with almost primes. In a related but different direction, an important method known as the affine linear sieve was introduced and developed by J. Bourgain, A. Gamburd, and P. Sarnak in \cite{BGS}, which established the existence of almost prime solutions to certain quadratic equations in \cite{LS}. We refer the reader to \cite{BGS} and \cite{LS}, and also a short discussion of this work in \cite[Section 1]{CM}, for more detailed information on this topic.

The main result of this paper improves on the bound on the number of prime factors in Theorem \ref{thm SS} with a modest cost on
$\textnormal{codim } V_{F}^*$. %(in comparison to the requirement in Theorem \ref{thm CM}).
In fact we establish a result analogous to Theorem \ref{thm CM} for semiprimes, which are natural numbers with precisely two (not necessarily distinct) prime factors, with an exponential lower bound for $\textnormal{codim } V_{F}^*$.

\begin{thm}
\label{thm main1}
Let $F(\mathbf{x}) \in \mathbb{Z}[x_1, \ldots, x_n]$ be a degree $d$ homogeneous form.
Suppose that $F$ satisfies the local conditions \textnormal{($\star$)} and
$\textnormal{codim } V_{F}^* > 4^d \cdot 8(2d-1).$
Then the equation (\ref{eqn main}) has an infinite number of solutions $(\ell_1, \ldots, \ell_n)$
such that $\ell_i$ has precisely two (not necessarily distinct) prime factors for each $1 \leq i \leq n$.
\end{thm}

We note that a more general result, Theorem \ref{thm main 2}, is proved in this paper,
where we obtain quantitative estimates on the number of semiprime solutions of a specific shape, from which Theorem \ref{thm main1} follows immediately.
We present this theorem in Section \ref{sec final}.
%We chose to postpone stating the result to Section \ref{sec final} in order to keep the introduction concise.
The proof is based on several key observations. The first observation is that
solving the equation (\ref{eqn main})
in semiprimes is equivalent to solving the equation
\begin{equation}
\label{eqn bihmg}
F(x_1 y_1, \ldots, x_n y_n) = 0
\end{equation}
in primes. This observation appears to be not particularly helpful at first because
the only known result for solving general polynomial equations in primes is Theorem \ref{thm CM}. However,
we observe that $F(x_1y_1, \ldots, x_n y_n)$ is now a bihomogeneous form (defined in Section \ref{sec prelim}), and we can in fact exploit this structure to obtain an estimate on the number of prime solutions to (\ref{eqn bihmg}) efficiently. We employ the work of D. Schindler \cite{DS} on bihomogeneous
forms to achieve this. Therefore, we do not rely on the sophisticated method of B. Cook and \'{A}. Magyar \cite{CM} which would drive up the requirement for $\textnormal{codim } V_{F}^*$.
In particular, our method avoids the use of sieve theory, unlike the work of \cite{MT}. %and \cite{SS}.
Another observation is that the dimensions of the variants (defined in (\ref{defnsinglocBHG})) of the singular locus
%defined in Section \ref{sec prelim}
of $\{ (\mathbf{x}, \mathbf{y}) \in \mathbb{C}^{2n} : F(x_1y_1, \ldots, x_n y_n) = 0\}$ are well-controlled by $\dim V_{F}^*$ (Theorem \ref{mainineqthm}), and this plays a crucial role in the proof of Theorem \ref{thm main 2}.

We remark that Theorem \ref{thm SS} was improved recently by D. Schindler and E. Sofos in \cite{SS}.
As a special case of their main result \cite[Theorem 1.1]{SS}, D. Schindler and E. Sofos established \cite[Corollary 1.2]{SS}, which holds when $F$ is non-singular, $d \geq 5$, and $n > 2^{d-1} (d^2 - 1)$, from which one can obtain a quantitative estimate on the number of solutions to the equation (\ref{eqn main}) whose coordinates have at most $O(d \log n/ (\log \log n) )$ prime factors. Their approach is based on combining sieve methods and the Hardy-Littlewood circle method. Note we have stated this result by D. Schindler and E. Sofos and Theorem \ref{thm SS}
in terms of the number of prime factors, but in fact the results were obtained in terms of the smallest prime divisors. Thus they obtained results for a different problem from which the mentioned statements follow immediately.

The organization of the rest of the paper is as follows. We devote Sections \ref{sec prelim}, \ref{sec minor arcs}, and \ref{sec major arcs}
to establishing Theorem \ref{thm main bihmg}, which is of interest on its own, regarding the number of prime solutions to systems of bihomogeneous equations.
This is achieved by the Hardy-Littlewood circle method. We cover preliminaries in Section \ref{sec prelim}, and obtain
the minor arcs estimate in Section \ref{sec minor arcs} and the major arcs estimate in Section \ref{sec major arcs}.
In Section \ref{sec final}, we establish the main results of this paper by using estimates obtained in the previous sections.
%For the sake of completion in Appendix \ref{appendix} we provide proof for Lemma \ref{Lemma 4.3 in DS} which is a generalization of Lemma \cite[Lemma 4.3]{DS}.

%Throughout the paper we do not distinguish between the two terms `homogeneous polynomial' and `form', and we will be using these terms interchangeably.
We use the well-known notation $\ll$ and $\gg$ of Vinogradov.
By an affine variety we mean an algebraic set which is not necessarily irreducible.
We use the notation $e(x)$ to denote $e^{2\pi i x}$.
%For $\mathbf{x} = (x_1, \ldots, x_n)$, the notation
%$$
%\sum_{\mathbf{x} \in [0,X]^n}
%$$
%means we are summing over all $\mathbf{x} \in \mathbb{Z}^n$ with $0 \leq x_i \leq X \ (1 \leq i \leq n)$.
We let $\mathbbm{1}_{H}$ be the characteristic function of the set $H$.
%For $q \in \mathbb{N}$ we use the numbers from $\{0, 1, \ldots, q-1 \}$ to represent the residue classes of $\mathbb{Z}/q\mathbb{Z}$.
Given $\omega_1, \ldots, \omega_{h_0} \in \mathbb{C}[x_1, \ldots, x_n]$, we let
$V(\omega_1, \ldots, \omega_{h_0}) = \{ \mathbf{z} \in \mathbb{C}^n : \omega_i (z_1, \ldots, z_n) = 0  \ (1 \leq i \leq h_0) \}$.
%We abuse the notation slightly and sometimes use $\mathbb{Z}/q\mathbb{Z}$ to mean the set of integers $\{0, 1, \ldots, q-1 \}$; however, it should be clear from the context what is meant.
%Finally, given $\mathbf{x} = (x_1, \ldots, x_n)$ we let $|\mathbf{x}| = \max_{1 \leq i \leq n} |x_i|$.

\section{Preliminaries}
\label{sec prelim}
We set some notation to be used throughout Sections \ref{sec prelim}, \ref{sec minor arcs}, and \ref{sec major arcs}.
Let $\mathbf{x} = (x_1, \ldots, x_{n_1})$ and $\mathbf{y} = (y_1, \ldots, y_{n_2})$.
We consider the following degree $(d_1 + d_2)$ polynomials with integer coefficients
\begin{equation}
\label{def g}
g_1(\mathbf{x}; \mathbf{y}), \ldots, g_R(\mathbf{x}; \mathbf{y}),
\end{equation}
which will be referred to as $\mathbf{g}$. We denote the homogeneous degree $(d_1 + d_2)$ portion of these polynomials
as $G_1(\mathbf{x}; \mathbf{y}), \ldots, G_R(\mathbf{x}; \mathbf{y})$ respectively, which will be referred to as $\mathbf{G}$.
We further assume that each $G_r(\mathbf{x}; \mathbf{y})$ is bihomogeneous of bidegree $(d_1, d_2)$, in other words
$$
G_r( s x_1, \ldots, s x_{n_1} ; t y_1, \ldots, t y_{n_2} ) = s^{d_1} t^{d_2} G_r(\mathbf{x} ; \mathbf{y} ).
$$
We also assume $d_1, d_2 > 1$.
%Throughout the paper we assume $d_1, d_2 > 1$.
%We also assume at least one of $d_1$ or $d_2$ is strictly greater than $1$.

Let $\wp$ denote the set of primes. Let $\Lambda^*(x) = \log x$ if $x \in \wp$ and $0$ otherwise.
We let $\Lambda^*(\mathbf{x}) = \Lambda^*(x_1) \cdots \Lambda^*(x_{n_1})$ and similarly for $\Lambda^*(\mathbf{y})$.
Let us define
\begin{equation}
\mathcal{N}_{\wp}(\mathbf{g}; P_1, P_2 ) = \sum_{\mathbf{x} \in [0,P_1]^{n_1}} \sum_{\mathbf{y} \in [0,P_2]^{n_2}}
\Lambda^*(\mathbf{x}) \Lambda^*(\mathbf{y}) \ \mathbbm{1}_{V(g_1, \ldots, g_R)} (\mathbf{x}, \mathbf{y}),
\end{equation}
which is the number of prime solutions
$(\mathbf{x}, \mathbf{y}) \in ([0, P_1]^{n_1} \times [0, P_2]^{n_2}) \cap \wp^{n_1 + n_2}$ to the system of equations
\begin{equation}
\label{set of eqn 3}
g_{r}(\mathbf{x}; \mathbf{y}) = 0 \ \ (1 \leq r \leq R)
\end{equation}
counted with weight $\Lambda^*(\mathbf{x}) \Lambda^*(\mathbf{y})$.
Without loss of generality we assume $P_1 \geq P_2$.

Let us define the following matrices
$$
\textnormal{Jac}_{\mathbf{G}, 1} = \left( \frac{\partial G_r}{\partial x_j}\right)_{\substack{ 1\leq r \leq R \\ 1\leq j\leq n_1}}
\ \
\text{  and  }
\ \
\textnormal{Jac}_{\mathbf{G}, 2} = \left(\frac{\partial G_r}{\partial y_j}\right)_{\substack{1\leq r \leq R\\ 1 \leq j \leq n_2 }}.
$$
We introduce the following affine varieties in $\mathbb{A}_{\mathbb{C}}^{n_1 + n_2}$,
\begin{equation}
\label{defnsinglocBHG}
V_{\mathbf{G}, i}^* := \{ (\mathbf{x}, \mathbf{y})  \in  \mathbb{C}^{n_1 + n_2} :   %\mathbb{A}_{\mathbb{C}}^{n_1 + n_2} :
\textnormal{rank} (\textnormal{Jac}_{\mathbf{G}, i}) < R \}
\ \ (i=1,2).
\end{equation}
We define them in a similar manner for other systems of bihomogeneous forms as well.

We devote Sections \ref{sec prelim}, \ref{sec minor arcs}, and \ref{sec major arcs} to proving the following theorem.
\begin{thm}
\label{thm main bihmg}
Let $\mathbf{g}$ be as in (\ref{def g}), $P = P_1^{d_1} P_2^{d_2}$, and $1 \leq \mathfrak{b} = \frac{\log P_1}{\log P_2}$. Suppose
\begin{equation}
\label{assmpn on F}
\textnormal{codim } V_{\mathbf{G}, i}^* > 2^{d_1 + d_2} \max\{  2R(R+1)(d_1 + d_2 -1), R(\mathfrak{b}d_1 + d_2)  \} \ \ \ (i=1,2).
\end{equation}
Then there exists $c>0$ such that the following holds
$$
\mathcal{N}_{\wp}(\mathbf{g}; P_1, P_2 ) = \sigma_{\mathbf{g}} P_1^{n_1 - d_1 R} P_2^{n_2 - d_2 R} + O\left( \frac{P_1^{n_1 - d_1 R} P_2^{n_2 - d_2 R} }{(\log P)^c}\right).
$$
Furthermore, $\sigma_{\mathbf{g}} > 0$ provided the system of equations (\ref{set of eqn 3}) has a non-singular solution in $(\mathbb{Z}_p^{\times})^{n_1 + n_2}$ for each prime $p$ and the system $G_{r}(\mathbf{x}; \mathbf{y}) = 0 \ (1 \leq r \leq R)$ has a non-singular real solution in $(0,1)^{n_1+n_2}$.
\end{thm}

We establish Theorem \ref{thm main bihmg} by an application of the Hardy-Littlewood circle method.
%Let $\mathbb{T} = \mathbb{R} / \mathbb{Z}$ and $\| \beta \|$ denote the distance from $\beta \in \mathbb{R}$ to the nearest integer, which induces a metric on $\mathbb{T}$ via $d(\alpha, \beta) = \| \alpha - \beta \|$.
Let $P = P_1^{d_1}P_2^{d_2}$. We define the \textit{major arcs} $\mathfrak{M}(\vartheta)$ to be the set of points
$\boldsymbol{\alpha} = (\alpha_1, \ldots, \alpha_R) \in [0,1)^{R}$ satisfying the following: there exist $1 \leq q \leq P^{R(d_1 + d_2 -1) \vartheta}$ and
$a_1, \ldots, a_R \in \mathbb{Z}$ with
$$
\gcd(q, a_1, \ldots, a_R) = 1 \ \ \ \text{  and  } \ \ \ \  2 | q \alpha_r - a_r | \leq P_1^{-d_1} P_2^{- d_2} P^{R(d_1 + d_2 -1) \vartheta} \ (1 \leq r \leq R).
$$
We define the \textit{minor arcs} to be the complement $\mathfrak{m}(\vartheta) =  [0,1)^{R} \backslash \mathfrak{M}(\vartheta)$.

Let us define
\begin{equation}
\label{def S}
S(\boldsymbol{\alpha}) := \sum_{\mathbf{x} \in [0,P_1]^{n_1} } \sum_{\mathbf{y} \in [0, P_2]^{n_2}  }   \Lambda^*(\mathbf{x}) \Lambda^*(\mathbf{y}) \ e \left( \sum_{r=1}^R \alpha_r g_r(\mathbf{x} ; \mathbf{y}) \right).
\end{equation}
By the orthogonality relation, we have
\begin{eqnarray}
\label{orthog reln}
\mathcal{N}_{\wp}(\mathbf{g}; P_1, P_2 ) = \int_{[0,1)^R} S(\boldsymbol{\alpha}) \ \mathbf{d}\boldsymbol{\alpha}
= \int_{\mathfrak{M}(\vartheta') } S(\boldsymbol{\alpha}) \ \mathbf{d}\boldsymbol{\alpha} + \int_{\mathfrak{m}(\vartheta') } S(\boldsymbol{\alpha}) \ \mathbf{d}\boldsymbol{\alpha}.
\end{eqnarray}
For a suitable choice of $\vartheta'$, we prove estimates for the integral over the minor arcs in Section \ref{sec minor arcs} and
over the major arcs in Section \ref{sec major arcs}. In this section, we collect results to set up the proof for these estimates.

We make frequent use of the following basic lemma on the dimensions of affine varieties.
\begin{lem}
\label{lemma on dim}
Let $X$ be an irreducible affine variety in $\mathbb{A}_{\mathbb{C}}^n$, and let $\omega \in \mathbb{C}[x_1, \ldots, x_n]$. Suppose $\emptyset \not = X \cap V(\omega)$ and  $X \not \subseteq V(\omega)$.
Then every irreducible component of $X \cap V(\omega)$ has dimension $(\dim X - 1)$.
Furthermore, if $Y = \cup_{1 \leq i \leq s_0} Y_i$ and $Z = \cup_{1 \leq j \leq t_0} Z_j$ are affine varieties in $\mathbb{A}_{\mathbb{C}}^n$, where $Y_i$'s and $Z_j$'s are the irreducible components of $Y$ and $Z$ respectively, such that $\emptyset \not = Y_i \cap Z_j$ $(1 \leq i \leq s_0, 1 \leq j \leq t_0)$, then
$
\dim Z  - \textnormal{codim } Y \leq \dim (Z \cap Y).
$
\end{lem}
\begin{proof}
The first part of the statement is precisely \cite[Exercise I.1.8]{H}. For the second part we recall \cite[Proposition I.7.1]{H}:
If $V$ and $W$ are irreducible affine varieties in $\mathbb{A}_{\mathbb{C}}^n$ and $V \cap W \not = \emptyset$, then
$\dim (V \cap W) \geq \dim V+ \dim W - n.$ The second part of the statement follows immediately from this result, and we leave the details to the reader.
\end{proof}
Let us also recall that given an affine variety $X$ in $\mathbb{A}^n_{\mathbb{C}}$, if $X$ is defined by homogeneous polynomials then
every irreducible component of $X$ contains $\mathbf{0}$.
We prove the following lemma regarding $ \textnormal{codim } V_{\mathbf{G}, i}^*$, the codimension of $V_{\mathbf{G},i}^*$ as a subvariety of $\mathbb{A}_{\mathbb{C}}^{n_1 + n_2}$. %We remark that although
%the statement of the lemma is with respect to $x_1$ the lemma in fact works with any of the $\mathbf{x}$ or the $\mathbf{y}$ variables.
\begin{lem}
\label{Lemma on the B rank}
Let $G_1(\mathbf{x}; \mathbf{y}), \ldots, G_R(\mathbf{x}; \mathbf{y}) \in \mathbb{Z}[x_1, \ldots, x_{n_1}, y_1, \ldots, y_{n_2}]$ be bihomogeneous of bidegree $(d_1, d_2)$. Let $0 \leq s < n_1$ and $0 \leq t < n_2$. For each $1 \leq r \leq R$, let
$$
\mathfrak{F}_{r}( x_{s + 1}, \ldots, x_{n_1};  y_{t + 1}, \ldots, y_{n_2} ) =
{G}_{r}( 0, \ldots, 0,  x_{s + 1}, \ldots, x_{n_1};  0, \ldots, 0,  y_{t + 1}, \ldots, y_{n_2} ).
$$
Then we have
$$
\min \{ \textnormal{codim } V_{\boldsymbol{\mathfrak{F}}, 1}^*, \textnormal{codim }V_{\boldsymbol{\mathfrak{F}}, 2}^* \}
\geq \min \{ \textnormal{codim } V_{\mathbf{G}, 1}^*, \textnormal{codim } V_{\mathbf{G}, 2}^* \} - (s + t) (R + 1).
$$
\end{lem}
\begin{proof}
We consider the case $s = 1$ and $t = 0$ as the general case follows by repeating the argument for this case. It is clear from the definition that
$\textnormal{Jac}_{\boldsymbol{\mathfrak{F}}, 1}$ is obtained by removing the first column from $\textnormal{Jac}_{\mathbf{G}, 1} |_{x_1 = 0}$.
%the first column.
Let $W$ be the affine variety in $\mathbb{A}^{n_1 - 1+ n_2}_{\mathbb{C}}$ defined by the entries of the first column of $\textnormal{Jac}_{\mathbf{G}, 1}|_{x_1 = 0}$. In particular, $W$ is defined by $R$ homogeneous polynomials, and hence $\textnormal{codim } W \leq R$.
Let $\lambda_1(\mathbf{x}, \mathbf{y}), \ldots, \lambda_{K_1}(\mathbf{x}, \mathbf{y})$ denote
the determinants of matrices formed by $R$ columns of $\textnormal{Jac}_{\mathbf{G}, 1}$.
Then we see that $V_{\mathbf{G}, 1}^*$ is defined by these polynomials.
Take a point
$$
(0, \widetilde{\mathbf{x}}_0, \mathbf{y}_0) = (0, x_{0,2}, \ldots, x_{0,n_1}, y_{0,1}, \ldots, y_{0, n_2}) \in \{ x_1 \in \mathbb{C} : x_1 = 0 \} \times ( V_{\boldsymbol{\mathfrak{F}}, 1}^* \cap W).
$$
Let $1 \leq k \leq K_1$. Suppose $\lambda_{k} (\mathbf{x}, \mathbf{y})$ corresponds to $R$ columns of $\textnormal{Jac}_{\mathbf{G}, 1}$
which contains the first column. Then since every entry of the first column of $\textnormal{Jac}_{\mathbf{G}, 1}$ is $0$ at $(0, \widetilde{\mathbf{x}}_0, \mathbf{y}_0)$,
we have $\lambda_{k} (0, \widetilde{\mathbf{x}}_0, \mathbf{y}_0) = 0$. On the other hand, suppose $\lambda_{k} (\mathbf{x}, \mathbf{y})$ corresponds to a collection of $R$ columns which does not contain the first column. In this case $\lambda_{k} (0, x_2, \ldots,x_{n_1}, \mathbf{y})$ is the determinant of
one of the matrices formed by taking $R$ columns of $\textnormal{Jac}_{\boldsymbol{\mathfrak{F}}, 1}$, and hence $( \widetilde{\mathbf{x}}_0, \mathbf{y}_0)$ is a zero of this polynomial. Thus we have $\lambda_{k} (0, \widetilde{\mathbf{x}}_0, \mathbf{y}_0) = 0$ in this case as well.
Therefore, we have shown that
$$
\{ \mathbf{0} \} \subseteq \{ x_1 \in \mathbb{C} :  x_1 = 0 \} \times (V_{\boldsymbol{\mathfrak{F}}, 1}^* \cap W)   \subseteq V_{\mathbf{G}, 1}^*
\cap  V (x_1)  \subseteq \mathbb{A}_{\mathbb{C}}^{n_1 + n_2}.
$$
We know that $\dim  (V_{\mathbf{G}, 1}^*  \cap V(x_1)  )$ is either $(\dim V_{\mathbf{G}, 1}^* - 1)$ or $\dim V_{\mathbf{G}, 1}^*$.
By Lemma \ref{lemma on dim} we obtain $ \dim V_{\boldsymbol{\mathfrak{F}}, 1}^* - R \leq \dim V_{\mathbf{G}, 1}^*,$
and consequently
$
\textnormal{codim } V_{\boldsymbol{\mathfrak{F}}, 1}^* \geq \textnormal{codim } V_{\mathbf{G}, 1}^* - (R + 1).
$

Next we consider the case $i=2$. In this case $\textnormal{Jac}_{\boldsymbol{\mathfrak{F}}, 2}$ is obtained by setting $x_1 = 0$ in $\textnormal{Jac}_{\mathbf{G}, 2}$. Thus we have
$$
\{ \mathbf{0} \} \subseteq  \{ x_1 \in \mathbb{C} : x_1 = 0 \} \times V_{\boldsymbol{\mathfrak{F}}, 2}^* \subseteq V_{\mathbf{G}, 2}^*  \cap V(x_1)  \subseteq \mathbb{A}_{\mathbb{C}}^{n_1 + n_2}.
$$
Therefore, it follows that
$\dim V_{\boldsymbol{\mathfrak{F}}, 2}^* \leq \dim  V_{\mathbf{G}, 2}^*$,
and consequently we have
$
\textnormal{codim } V_{\boldsymbol{\mathfrak{F}}, 2}^* \geq \textnormal{codim } V_{\mathbf{G}, 2}^* - 1.
$
Our result is then immediate.
\end{proof}

%Recall the definition of $S(\boldsymbol{\alpha})$ given in (\ref{def S}).
By applying Cauchy-Schwarz inequality we obtain
\begin{eqnarray}
|S(\boldsymbol{\alpha})|^2
\notag
%\\
%&\ll&  P_1^{n_1} \sum_{\mathbf{x} \in [0,P_1]^{n_1} } \
%\sum_{\mathbf{y}, \mathbf{y}' \in [0,P_2]^{n_2} }   \Lambda^*(y_1) \cdots  \Lambda^*(y_{n_1}) \Lambda^*(y'_1) \cdots  \Lambda^*(y'_{n_2})
%\ e \left( \sum_{r=1}^R  \alpha_r \cdot (f_r(\mathbf{x} ; \mathbf{y} ) -  f_r(\mathbf{x} ; \mathbf{y}' ) )  \right)
%\notag
\ll (\log P_1)^{n_1} P_1^{n_1}  \sum_{\mathbf{y}, \mathbf{y}' \in [0,P_2]^{n_2} }  \Lambda^*(\mathbf{y}) \Lambda^*( \mathbf{y}') \sum_{\mathbf{x} \in [0,P_1]^{n_1} } e \left( \sum_{r=1}^R  \alpha_r (g_r(\mathbf{x} ; \mathbf{y} ) -  g_r(\mathbf{x} ; \mathbf{y}' ) )  \right).
\notag
\end{eqnarray}
We then apply Cauchy-Schwarz inequality once more and obtain
\begin{eqnarray}
\label{sum1}
|S(\boldsymbol{\alpha})|^4
\ll (\log P_1)^{2n_1} (\log P_2)^{2n_2} P_1^{2 n_1} P_2^{2n_2} \sum_{\mathbf{x}, \mathbf{x}' \in [0,P_1]^{n_1} }  \  \sum_{\mathbf{y}, \mathbf{y}' \in [0,P_2]^{n_2} } \
e \left( \sum_{r=1}^R \alpha_r \cdot \mathfrak{d}_r(\mathbf{x}, \mathbf{x}' ; \mathbf{y}, \mathbf{y}' )  \right),
\end{eqnarray}
where
\begin{equation}
\label{defn of g}
\mathfrak{d}_r(\mathbf{x}, \mathbf{x}' ; \mathbf{y}, \mathbf{y}' ) = g_r(\mathbf{x} ; \mathbf{y} ) -  g_r(\mathbf{x} ; \mathbf{y}' )
- g_r(\mathbf{x}' ; \mathbf{y} ) +  g_r(\mathbf{x}' ; \mathbf{y}'  ).
\end{equation}
In order to simplify our notation we denote $\mathbf{u} = (\mathbf{x}, \mathbf{x}')$ and $\mathbf{v} = (\mathbf{y}, \mathbf{y}')$,
%$\mathcal{I}_1 = \mathcal{B}_1 \times \mathcal{B}_1$,  $\mathcal{I}_2 = \mathcal{B}_2 \times \mathcal{B}_2$,
and write the sum on the right hand side of (\ref{sum1}) as
$$
T(\boldsymbol{\alpha}) := \sum_{\mathbf{u} \in [0, P_1]^{2n_1} }  \  \sum_{\mathbf{v} \in [0, P_2]^{2 n_2} } \
e \left(  \sum_{r=1}^R  \alpha_r \mathfrak{d}_r(\mathbf{u} ; \mathbf{v} )  \right).
$$
It is clear from the definition of the polynomial $\mathfrak{d}_r(\mathbf{u}; \mathbf{v})$ given in (\ref{defn of g}) that
it is a degree $(d_1 + d_2)$ polynomial (in $\mathbf{u}$ and $\mathbf{v}$) whose homogeneous degree $(d_1 + d_2)$ portion is
$$
\mathfrak{D}_r(\mathbf{u}; \mathbf{v}) = \mathfrak{D}_r(\mathbf{x}, \mathbf{x}' ; \mathbf{y}, \mathbf{y}' ) = G_r(\mathbf{x} ; \mathbf{y} ) -  G_r(\mathbf{x} ; \mathbf{y}' )
- G_r(\mathbf{x}' ; \mathbf{y} ) +  G_r(\mathbf{x}' ; \mathbf{y}').
$$
It is then immediate that $\mathfrak{D}_r(\mathbf{u}; \mathbf{v})$ is a bihomogeneous form  of bidegree $(d_1, d_2)$.

Note we have
\begin{equation}
\label{deriv of G-1}
\frac{\partial  \mathfrak{D}_r}{\partial x_j} (\mathbf{x}, \mathbf{x}' ; \mathbf{y}, \mathbf{y}' ) =
\frac{\partial G_r}{\partial x_j} (\mathbf{x} ; \mathbf{y} ) -  \frac{\partial G_r}{\partial x_j} (\mathbf{x} ; \mathbf{y}' ).
\end{equation}
%and
%$$
%\frac{\partial}{\partial x'_j} \mathfrak{D}_r(\mathbf{x}, \mathbf{x}' ; \mathbf{y}, \mathbf{y}' ) =
%- \frac{\partial}{\partial x'_j} G_r(\mathbf{x}' ; \mathbf{y} ) +  \frac{\partial}{\partial x'_j} G_r(\mathbf{x}' ; \mathbf{y}' ).
%$$
Let $M_1$ be the matrix obtained by removing $n_1$ columns corresponding to $\mathbf{x}'$ (that is $(n_1 + 1)$-th column to $(2 n_1)$-th column) from $\textnormal{Jac}_{\boldsymbol{\mathfrak{D}}, 1}$. It is clear from (\ref{deriv of G-1}) that $M_1$ is independent of $\mathbf{x}'$. Let
$
V_{M_1}^* = \left\{ (\mathbf{x}, \mathbf{y}, \mathbf{y}' ) \in \mathbb{C}^{n_1 + 2 n_2} :  \textnormal{rank } M_1 < R   \right\}.
$
Since $M_1 |_{\mathbf{y}' = \mathbf{0} }$ is precisely $\textnormal{Jac}_{\mathbf{G}, 1}$, we have
$(\mathbf{x}, \mathbf{y}) \in V^*_{\mathbf{G}, 1}$ if and only if $( \mathbf{x}, \mathbf{y}, \mathbf{0}) \in V_{M_1}^*$.
Therefore, we see that
$$
V_{\mathbf{G}, 1}^* \times \{ \mathbf{y}' \in \mathbb{C}^{n_2} :  \mathbf{y}' = \mathbf{0} \} \times \{ \mathbf{x}' \in  \mathbb{C}^{n_1} \}   = (V_{M_1}^*  \times \{ \mathbf{x}' \in  \mathbb{C}^{n_1} \}) \cap V(y'_1, \ldots, y'_{n_2}) \subseteq \mathbb{A}_{\mathbb{C}}^{2n_1+ 2 n_2}.
$$
Let $W = \{ (\mathbf{x}, \mathbf{y}, \mathbf{y}', \mathbf{x}') \in \mathbb{C}^{2 n_1 + 2 n_2} : (\mathbf{x}, \mathbf{x}', \mathbf{y}, \mathbf{y}') \in V_{\boldsymbol{\mathfrak{D}},1}^{*} \}$. Then $\dim W = \dim V_{\boldsymbol{\mathfrak{D}},1}^{*}$.
Since $M_1$ is a submatrix of $\textnormal{Jac}_{\boldsymbol{\mathfrak{D}}, 1}$ we have
$%V_{\boldsymbol{\mathfrak{D}},1}^{*}
W \subseteq V_{M_1}^* \times  \{ \mathbf{x}' \in  \mathbb{C}^{n_1} \}.$ Therefore, it follows that
$$
\{ \mathbf{0} \} \subseteq W \cap V(y'_1, \ldots, y'_{n_2} ) \subseteq V_{\mathbf{G}, 1}^*  \times \{ \mathbf{y}' \in \mathbb{C}^{n_2} : \mathbf{y}' = \mathbf{0} \}\times \{ \mathbf{x}' \in  \mathbb{C}^{n_1} \}   \subseteq \mathbb{A}_{\mathbb{C}}^{2 n_1+ 2 n_2}.
$$
Consequently, by Lemma \ref{lemma on dim} we obtain $\dim V_{\boldsymbol{\mathfrak{D}}, 1}^*  - n_2 \leq n_1 + \dim  V_{ \mathbf{G},  1}^*$, which is equivalent to
\begin{eqnarray}
\label{new2.11}
\textnormal{codim } V_{\boldsymbol{\mathfrak{D}}, 1}^*  = (2n_1 + 2 n_2) - \dim V_{\boldsymbol{\mathfrak{D}}, 1}^* \geq n_1 + n_2 - \dim V_{\mathbf{G}, 1}^* = \textnormal{codim } V_{\mathbf{G}, 1}^*.
\end{eqnarray}
By reversing the roles of $\mathbf{x}$ and $\mathbf{x}'$ with that of $\mathbf{y}$ and $\mathbf{y}'$, we also obtain
$
\textnormal{codim } V_{\boldsymbol{\mathfrak{D}}, 2}^*  \geq \textnormal{codim } V_{\mathbf{G}, 2}^*.
$
Therefore, it follows from (\ref{assmpn on F}) that
$$
\textnormal{codim } V_{\boldsymbol{\mathfrak{D}}, i}^*  > 2^{d_1 + d_2} \max\{  2 R(R+1)(d_1 + d_2 -1), R( \mathfrak{b} d_1 + d_2)  \} \ \ \ (i = 1, 2).
$$

Let $\delta_0 > 0$ be a sufficiently small constant. We now define the following constant
\begin{equation}
\label{def K}
K := \frac{ \min\{  \textnormal{codim } V_{\boldsymbol{\mathfrak{D}}, 1}^*, \textnormal{codim } V_{\boldsymbol{\mathfrak{D}}, 2}^*  \} - \delta_0 }{2^{d_1 + d_2 - 2}}.
\end{equation}
In particular, we have
\begin{equation}
\label{lower bdd on K}
K  > 4 \max \{ 2 R (R+1)(d_1 + d_2 -1),  R(\mathfrak{b} d_1 + d_2) \}.
\end{equation}
% To see this first make sure $\delta_0$ is sufficiently small so that
% $\textnormal{codim } V_{\boldsymbol{\mathfrak{D}}, i}^* - \delta_0  > 2^{d_1 + d_2} \max\{  2 R(R+1)(d_1 + d_2 -1), R( \mathfrak{b} d_1 + d_2)  \}$
%
%

We make use of the following generalization of \cite[Lemma 4.3]{DS} which gives us an exponential sum estimate on the minor arcs.
We remark that owing to a minor oversight in \cite[pp. 498]{DS}, the presence of $\delta_0$ in the statement is necessary.
Since the lemma can be obtained by following the argument of \cite[Lemma 4.3]{DS} in our setting, we omit the details.
We shall refer to $\mathfrak{B} \subseteq \mathbb{R}^m$ as a box, if $\mathfrak{B}$ is of the form
$
\mathfrak{B} = I_1 \times \cdots \times I_m,
$
where each $I_j$ is a closed or open or half open/closed interval $(1 \leq j \leq m)$.
\begin{lem}\cite[Lemma 4.3]{DS}
\label{Lemma 4.3 in DS}
Let $\mathbf{u} = (u_1, \ldots, u_{m_1})$ and $\mathbf{v} = (v_1, \ldots, v_{m_2})$. Let $\mathfrak{B}_i \subseteq \mathbb{R}^{m_i}$ be a box with sides $\leq 1$
$(i=1,2)$. Let $\mathfrak{f}_1(\mathbf{u}; \mathbf{v}), \ldots,$ $\mathfrak{f}_R(\mathbf{u}; \mathbf{v})$ be degree $(d_1 + d_2)$ polynomials with rational coefficients and let their degree $(d_1 + d_2)$ homogeneous portions be $\mathfrak{F}_1(\mathbf{u}; \mathbf{v}), \ldots, \mathfrak{F}_R(\mathbf{u}; \mathbf{v})$
respectively. For each $1 \leq r \leq R$, suppose $\mathfrak{F}_r(\mathbf{u}; \mathbf{v})$ is a bihomogeneous form of bidegree $(d_1, d_2)$ with integer coefficients. Let $\delta_0 > 0$ be a sufficiently small constant. Let $P = P_1^{d_1} P_2^{d_2}$, $1 \leq \mathfrak{b} = \frac{\log P_1}{ \log P_2}$, $0 < \vartheta \leq (\mathfrak{b} d_1 + d_2)^{-1}$, and
$$
\widetilde{K} = \frac{\min \left\{ \textnormal{codim}
(V^*_{\boldsymbol{ \mathfrak{F} }, 1}),  \textnormal{codim}
(V^*_{\boldsymbol{ \mathfrak{F} }, 2} ) \right\} - \delta_0 }{ 2^{d_1 + d_2 - 2} }.
$$
Consider the exponential sum
$$
\widetilde{T}(\boldsymbol{\alpha}) = \sum_{\mathbf{u} \in P_1 \mathfrak{B}_1 }  \  \sum_{\mathbf{v} \in P_2 \mathfrak{B}_2  } \
e \left( \sum_{r=1}^R \alpha_r \mathfrak{f}_r(\mathbf{u} ; \mathbf{v} )  \right).
$$
Then we have either
$$
\textnormal{(i)  } \boldsymbol{\alpha} \in \mathfrak{M}(\vartheta) \ \ \
\textnormal{  or  } \ \ \
\textnormal{(ii)  } |\widetilde{T}(\boldsymbol{\alpha})| \ll P_1^{m_1} P_2^{m_2} P^{- \widetilde{K} \vartheta} (\log P)^{m_1}.
$$
Here the implicit constant is independent of $\vartheta$, and it is also independent of the coefficients of $(\mathfrak{f}_r(\mathbf{u}; \mathbf{v}) - \mathfrak{F}_r(\mathbf{u}; \mathbf{v}))$ for each $1 \leq r \leq R$.
\end{lem}
We remark that the hypotheses in the statement of Lemma \ref{Lemma 4.3 in DS} are sufficient and the additional assumption \cite[lines 1-2, pp.488]{DS}
is in fact unnecessary; this can be verified by going through the proof of \cite[Lemma 4.3]{DS} and observing that the expression in \cite[line 22, pp.496]{DS}
is a multilinear form with integer coefficients due to the factor $d_1! d_2!$ as long as $F_1, \ldots, F_R$ have integer coefficients.
We note the fact that the implicit constant is independent of the lower degree terms of $\mathfrak{f}_r(\mathbf{u}; \mathbf{v})$ becomes crucial when we apply this lemma in Section \ref{sec major arcs}. We have the following exponential sum estimate as a corollary which we also use in Section \ref{sec major arcs}.
\begin{cor}
\label{cor 1 of Lemma 4.3 in DS}
Make all the assumptions of Lemma \ref{Lemma 4.3 in DS}.
Suppose $\gcd(q, a_1, \ldots, a_R) = 1$.
Then for any $\varepsilon > 0$ we have
\begin{eqnarray}
\notag
\sum_{ \substack{ \mathbf{u} \in [0, q-1]^{m_1} \\  \mathbf{v} \in [0, q-1]^{m_2} } } e \left( \sum_{r=1}^{R} \mathfrak{f}_{r} (\mathbf{u};\mathbf{v})  \cdot a_{r} /q  \right) &\ll& q^{m_1 + m_2 - \frac{\widetilde{K}}{ R(d_1 + d_2 -1)  } + \varepsilon }.
\end{eqnarray}
\end{cor}
\begin{proof}
See the proof of \cite[Lemma 5.5]{DS}.
\iffalse
We apply Lemma \ref{Lemma 4.3 in DS} with $\mathfrak{B}_1 = [0,1)^{m_1}$, $\mathfrak{B}_2 = [0,1)^{m_2}$, $P_1 = P_2 = q$, and $\boldsymbol{\alpha} = \mathbf{a}/q$.
Let $\varepsilon > 0$ and define $\vartheta = (1 - \varepsilon)/ ((d_1 + d_2) R (d_1 + d_2 -1))$.
%$$
%(d_1 + d_2) R (d_1 + d_2 -1) \vartheta = 1 - \varepsilon,
%$$
%for $\varepsilon > 0$.
Suppose there exist $\widetilde{a}_1, \ldots, \widetilde{a}_R$ and
$1 \leq \widetilde{q} \leq q^{(d_1 + d_2) R (d_1 + d_2 - 1) \vartheta}$
such that $\gcd(\widetilde{q}, \widetilde{a}_1, \ldots, \widetilde{a}_R) = 1$ and
$$
2 |\widetilde{q} \alpha_r - \widetilde{a}_r| \leq q^{-d_1} q^{- d_2} q^{ (d_1 + d_2)R (d_1 + d_2 - 1) \vartheta} \ \ (1 \leq r \leq R).
$$
Then we have
\begin{equation}
\label{ineq 2}
2 |\widetilde{q} a_r - q \widetilde{a}_r| \leq q q^{-d_1} q^{- d_2} q^{ 1 - \varepsilon } < 1.
\end{equation}
Since $\widetilde{q}, a_r, q,$ and  $\widetilde{a}_r$ are all integers, the inequality (\ref{ineq 2}) implies
%$$
%\widetilde{q} a_r - q \widetilde{a}_r = 0,
%$$
%or equivalently
$
a_r / q = \widetilde{a}_r / \widetilde{q}
$
for each $1 \leq r \leq R$.
However, this is a contradiction because $\gcd(q, {a}_1, \ldots, {a}_R) = 1$ and $1 \leq \widetilde{q} < q$.
Thus we are in the alternative (ii) of Lemma \ref{Lemma 4.3 in DS} which gives us the desired estimate.
\fi
\end{proof}

\section{The minor arcs estimate}
\label{sec minor arcs}
From the bound (\ref{sum1}) we have
$
|S(\boldsymbol{\alpha})|^4 \ll (\log P_1)^{2n_1} (\log P_2)^{2n_2} P_1^{2 n_1} P_2^{2n_2} \ T(\boldsymbol{\alpha}).
$
Thus the following is an immediate consequence of applying Lemma \ref{Lemma 4.3 in DS} to $T(\boldsymbol{\alpha})$.
\begin{lem}
\label{Lemma 4.3 in DS ver2}
Let $K$ be as in (\ref{def K}) and $0 < \vartheta \leq (\mathfrak{b} d_1 + d_2)^{-1}$.
Then we have either
$$
\textnormal{(i)  } \boldsymbol{\alpha} \in \mathfrak{M}(\vartheta) \ \ \  \textnormal{  or  } \ \ \  \textnormal{(ii)  } |S(\boldsymbol{\alpha})| \ll P_1^{n_1} P_2^{n_2} P^{- K \vartheta / 4} (\log P)^{n_1 + \frac{n_2}{2}}.
$$
Here the implicit constant is independent of $\vartheta$.
\end{lem}

We define
$$
\sigma := \frac12 \left( \frac{K}{4} -   \max \{ 2 R (R+1)(d_1 + d_2 -1),  R(\mathfrak{b} d_1 + d_2)    \} \right),
$$
which we know to be positive because of (\ref{lower bdd on K}).
Let us fix $\vartheta_0$ satisfying
\begin{equation}
\label{bdd on theta0}
0 < \vartheta_0 \leq (\mathfrak{b} d_1 + d_2)^{-1} \ \ \text{  and  } \ \ \frac{  \vartheta_0 K}{4} > R + \varepsilon_0
\end{equation}
for some $\varepsilon_0 > 0$ sufficiently small, which is possible because of (\ref{lower bdd on K}).

Let us set
$$
\zeta : = \frac{4 R (R+1)(d_1 + d_2 -1) + 4 \sigma }{ K },
$$
which can be verified to satisfy $0 < \zeta < 1$. %The inequality $\zeta < 1$ can be verified using
%$$
%\sigma \leq \frac12 \left( \frac{K}{4} -  R (R+1)(d_1 + d_2 -1) \right)
%$$
%and (\ref{lower bdd on K}).
Throughout Sections \ref{sec minor arcs} and \ref{sec major arcs} we let $C$ to be a sufficiently large positive constant which does not depend on $P$.
Let us define $\vartheta_{i+1} = \zeta \vartheta_i$ $(0 \leq i \leq M-1)$, where
$M$ is the smallest positive integer such that $P^{\vartheta_M} \leq (\log P)^{C}$.
From the definition of $M$ it follows that
%\begin{equation}
%\label{**'}
$
(\log P)^{C \zeta} < P^{\vartheta_M} = P^{ \zeta^{M} \vartheta_0 }, %\leq (\log P)^C,
$
%\end{equation}
for otherwise we have $P^{\vartheta_{M-1}} =  P^{\vartheta_M / \zeta} \leq (\log P)^{C}$ and this is a contradiction.
We then obtain $M \ll \log \log P.$
We also remark that from the definition of $M$ we have
\begin{equation}
\label{bdd on M-1}
(\log P)^{C} < P^{\vartheta_{M-1}}.
\end{equation}
%We simplify the bound (\ref{**'}) to obtain
%$
%M %< \frac{\log \left( \frac{C \zeta \log \log P}{\vartheta_0 \log P} \right)}{ \log \zeta }
%\ll \log \log P.
%$

Let us use the notation $ 0 \leq  \mathbf{a} \leq q$ to mean $0 \leq a_r \leq q$ $(1 \leq r \leq R)$.
The Lebesgue measure of $\mathfrak{M}( \vartheta_i)$ is bounded by the following quantity
\begin{eqnarray}
\textnormal{meas}(\mathfrak{M}( \vartheta_i) )
&\ll& \sum_{ q \leq P^{R(d_1 + d_2 - 1) \vartheta_i} } \
\sum_{ \substack{  0 \leq  \mathbf{a} \leq q  \\  \gcd(q, a_1, \ldots, a_R) = 1 } } q^{-R} P_1^{-d_1 R} P_2^{-d_2 R} P^{R^2 (d_1 + d_2 - 1) \vartheta_i }
\label{star}
\\
\notag
&\ll& P^{- R + R(R+1)(d_1 + d_2 - 1) \vartheta_i}.
\end{eqnarray}
Thus for each $0 \leq i \leq M-1$, we have by Lemma \ref{Lemma 4.3 in DS ver2} that
\begin{eqnarray}
\int_{ \mathfrak{M}( \vartheta_i) \backslash  \mathfrak{M}( \vartheta_{i+1}) } | S(\boldsymbol{\alpha}) | \ \mathbf{d} \boldsymbol{\alpha}
&\ll&
(\log P)^{n_1+ \frac{n_2}{2}} P_1^{n_1} P_2^{n_2} P^{-   \frac{K \vartheta_{i+1}}{4}  } \  \textnormal{meas}(\mathfrak{M}( \vartheta_i) )
\label{ineq 1}
\\
%&\ll&
%(\log P)^{n_1/2} P_1^{n_1} P_2^{n_2} P^{-  \frac{K \vartheta_{i+1}}{4}  - R + R(R+1)(d_1 + d_2 -1)\vartheta_{i}  }
%\notag
%\\
&\ll&
(\log P)^{n_1 + \frac{n_2}{2} } P_1^{n_1} P_2^{n_2} P^{- R - \sigma \theta_i},
\notag
\end{eqnarray}
where we obtained the final inequality using (\ref{star}), the relation $\vartheta_{i+1} = \zeta \vartheta_i$, and the definition of $\zeta$.
Since
$$
\mathfrak{m}(\vartheta_M) \subseteq \mathfrak{m}(\vartheta_0) \ \bigcup \  \bigcup_{0 \leq i \leq M-1} \mathfrak{M}( \vartheta_i) \backslash  \mathfrak{M}( \vartheta_{i+1}),
$$
it follows from Lemma \ref{Lemma 4.3 in DS ver2} with $\vartheta_0$ and (\ref{ineq 1}) that
\begin{eqnarray}
\int_{ \mathfrak{m}(\vartheta_M)} | S(\boldsymbol{\alpha}) | \ \mathbf{d} \boldsymbol{\alpha}
\notag
&\ll&
\int_{ \mathfrak{m}( \vartheta_0)} | S(\boldsymbol{\alpha}) | \  \mathbf{d} \boldsymbol{\alpha} + M \max_{0 \leq i \leq M-1} \int_{ \mathfrak{M}( \vartheta_i) \backslash  \mathfrak{M}( \vartheta_{i+1}) } | S(\boldsymbol{\alpha}) | \ \mathbf{d} \boldsymbol{\alpha}
\notag
\\
&\ll&
P_1^{n_1} P_2^{n_2} P^{- R - \frac{\varepsilon_0}{2} } + (\log \log P)  (\log P)^{n_1+ \frac{n_2}{2} } P_1^{n_1} P_2^{n_2} P^{- R - \sigma \theta_{M-1}  }
\notag
\\
&\ll&
(\log P)^{n_1 + \frac{n_2}{2} - \sigma C} (\log \log P) P_1^{n_1} P_2^{n_2} P^{- R},
\notag
\end{eqnarray}
where we obtained the final inequality using (\ref{bdd on M-1}).
Therefore, we have established the following.
\begin{prop}
\label{prop minor}
Given any $c>0$, %for sufficiently large $C>0$
we have
$$
\int_{\mathfrak{m}(\theta_M)} S(\boldsymbol{\alpha} ) \ \mathbf{d} \boldsymbol{\alpha}
\ll  \frac{ P_1^{n_1} P_2^{n_2} P^{- R}}{(\log P)^c}.
$$
%where $P^{\theta_M} = (\log P)^{C_0}$.
\end{prop}

\section{The major arcs estimate}
\label{sec major arcs}
%Let $C > 0$ be sufficiently large.
As the material in this section is fairly standard, we keep the details to a minimum and also refer the reader
to \cite[Sections 6 and 7]{CM} or \cite[Section 7]{Y} where similar work has been carried out.
Let us define $C_0$ by $P^{\vartheta_M} = (\log P)^{C_0}$. It is clear that $C_0$ depends on $P$; however, by the definition of $\theta_M$ we have
$C \zeta < C_0 \leq C$.
By the definition of $\mathfrak{M}(\vartheta_M)$ we can write
$$
\mathfrak{M}(\vartheta_M) = \bigcup_{1 \leq q \leq (\log P)^{C_0 R (d_1 + d_2 - 1)}} \bigcup_{\substack{ 0 \leq \mathbf{a} \leq q  \\ \gcd(q, a_1, \ldots, a_R) = 1}} \mathfrak{M}_{\mathbf{a}, q}(C_0),
$$
where
$$
\mathfrak{M}_{\mathbf{a}, q} (C_0) = \left\{ \boldsymbol{\alpha} \in [0,1)^R : 2 | q \alpha_r - a_r  | < \frac{(\log P)^{C_0 R (d_1 + d_2 - 1) }}{P} \ \ (1 \leq r \leq R)  \right\}.
$$
It can be verified that the arcs $\mathfrak{M}_{\mathbf{a}, q}(C_0)$'s are disjoint for $P$ sufficiently large.

We define
$$
\psi_{h}({t}) = \sum_{ \substack{ 0 \leq v \leq t  \\ v \equiv h (\text{mod } q) } } \Lambda^*(v).
$$
We use the notation $\mathbf{x} \equiv \mathbf{h}_1 \ (\text{mod } q)$ to mean $x_j \equiv h_{1,j} \ (\text{mod } q)$ for each $1 \leq j \leq n_1$,
and similarly for $\mathbf{y} \equiv \mathbf{h}_2 \ (\text{mod } q)$. We also denote $\mathbf{h} = (\mathbf{h}_1, \mathbf{h}_2)$.
Recall the definition of $S(\boldsymbol{\alpha})$ given in (\ref{def S}).
Let $\boldsymbol{\alpha} = \mathbf{a}/q + \boldsymbol{\beta} \in [0,1)^R$. In a similar manner as in \cite[(6.1)]{CM}, we can express $S(\boldsymbol{\alpha})$ as
\begin{eqnarray}
%\label{eqn for S1}
%&& \ \ \ S(\boldsymbol{\alpha})
%\\
%&=&
%\sum_{\mathbf{x} \in [0, P_1]^{n_1} } \sum_{\mathbf{y} \in [0, P_2]^{n_2} } \Lambda^*(\mathbf{x}) \Lambda^*(\mathbf{y}) \ e \left( \sum_{r=1}^{R} \alpha_{r} g_{r}(\mathbf{x}; \mathbf{y})  \right)
%\notag
%\\
%&=&
%\sum_{\substack{ \mathbf{h}_1 \in (\mathbb{Z}/ q\mathbb{Z})^{n_1} \\  \mathbf{h}_2 \in (\mathbb{Z}/ q\mathbb{Z})^{n_2}  } }
%\ \sum_{ \substack{ \mathbf{x} \in [0, P_1]^{n_1}   \\ \mathbf{x} \equiv \mathbf{h}_1  (\text{mod } q) } }
%\ \sum_{ \substack{ \mathbf{y} \in [0, P_2]^{n_2}    \\ \mathbf{y} \equiv \mathbf{h}_2  (\text{mod } q) } } \Lambda^*(\mathbf{x}) \Lambda^*(\mathbf{y}) \
%e \left( \sum_{r=1}^{R} a_{r} g_{r} (\mathbf{h}_1; \mathbf{h}_2 ) / q  \right)
% e \left( \sum_{r=1}^{R} \beta_{r} g_{r} (\mathbf{x}; \mathbf{y} )  \right)
%\notag
%\\
%&=&
\sum_{\mathbf{h} \in (\mathbb{Z}/ q\mathbb{Z})^{n_1 + n_2} } \
e \left( \sum_{r=1}^{R} a_{r} g_{r} (\mathbf{h}_1 ;  \mathbf{h}_2) / q \right)
\int_{ (\mathbf{t}_1, \mathbf{t}_2) \in [0, P_1]^{n_1} \times [0, P_2]^{n_2} } e \left( \sum_{r=1}^{R} \beta_{r} g_{r} (\mathbf{t}_1 ;  \mathbf{t}_2)  \right)
\ \mathbf{d} \boldsymbol{\psi}_{\mathbf{h}}(\mathbf{t}),
\notag
\end{eqnarray}
where $\mathbf{d} \boldsymbol{\psi}_{\mathbf{h}}(\mathbf{t})$ denotes the product measure
$$
d{\psi}_{{h}_{1,1} }({t}_1) \times \ldots \times d{\psi}_{{h}_{1, n_1} }({t}_{1,n_1}) \times d{\psi}_{{h}_{2,1}}({t}_{2,1}) \times \ldots \times d{\psi}_{{h}_{2, n_2} }({t}_{2,n_2}).
$$

Let $\phi$ be Euler's totient function. For a positive integer $q$, let $\mathbb{U}_q$ be the group of units in $\mathbb{Z}/q \mathbb{Z}$.
%where we let $\mathbb{U}_1 = \{ 0 \}.$
Let $\mathfrak{B}_0 = [0,1]^{n_1 + n_2}$ and
$$
\mathcal{I}( \mathfrak{B}_0,  \boldsymbol{\tau}) = \int_{(\mathbf{v}_1, \mathbf{v}_2 )\in \mathfrak{B}_0 } e \left( \sum_{r=1}^{R} \tau_{r}  \cdot G_{r}(\mathbf{v}_1 ; \mathbf{v}_2)  \right) \ \mathbf{d}\mathbf{v}.
$$
We denote $P_1^{d_1} P_2^{d_2} \boldsymbol{\beta} = (P_1^{d_1} P_2^{d_2} \beta_1, \ldots, P_1^{d_1} P_2^{d_2} \beta_R)$.
With these notations we have the following lemma.
\begin{lem}
\label{Lemma 6 in CM}
Let $c'>0$, $q \leq (\log P)^{C_0}$, and  $\mathbf{h} \in (\mathbb{Z} / q \mathbb{Z})^{n_1 + n_2}$. Suppose $\boldsymbol{\alpha} = \mathbf{a}/q + \boldsymbol{\beta} \in \mathfrak{M}_{\mathbf{a},q}(C_0)$.
Then we have
\begin{eqnarray}
&&\int_{(\mathbf{t}_1, \mathbf{t}_2) \in  [0, P_1]^{n_1} \times [0, P_2]^{n_2} } e \left( \sum_{r=1}^{R} \beta_{r} g_{r}(\mathbf{t}_1 ; \mathbf{t}_2)  \right)
\ \mathbf{d} \boldsymbol{\psi}_{\mathbf{h}}(\mathbf{t})
\notag
\\
&=& \mathbbm{1}_{ \mathbb{U}_q^{n_1 + n_2} } (\mathbf{h})  \  \frac{P_1^{n_1} P_2^{n_2}}{\phi(q)^{n_1 + n_2} } \ \mathcal{I}(\mathfrak{B}_0,  P_1^{d_1} P_2^{d_2} \boldsymbol{\beta}) + O( P_1^{n_1} P_2^{n_2} / (\log P)^{c'} ).
\notag
\end{eqnarray}
%where $\mathbbm{1}_{ \mathbf{h} \in (\mathbb{U}_q)^{n_1 + n_2} }$ is $1$ if $\mathbf{h} \in (\mathbb{U}_q)^{n_1 + n_2}$ and $0$ otherwise.
\end{lem}
We omit the proof of Lemma \ref{Lemma 6 in CM} because it can be established by following the argument of \cite[Lemma 6]{CM} in our setting and the changes required are minimal.

Let us define
$$
J(L) = \int_{ %| \boldsymbol{\tau} |  \leq  L
\boldsymbol{\tau}  \in [-L,L]^R
} \mathcal{I}( \mathfrak{B}_0,  \boldsymbol{\tau}) \ \mathbf{d} \boldsymbol{\tau}.
$$
It then follows by \cite[Lemma 5.6]{DS} that under our assumptions on $\mathbf{G}$, namely (\ref{assmpn on F}), we have
\begin{equation}
\label{defn of sing int}
\mu(\infty) = \int_{\boldsymbol{\tau} \in \mathbb{R}^R} \mathcal{I}( \mathfrak{B}_0,  \boldsymbol{\tau}) \ \mathbf{d} \boldsymbol{\tau},
\end{equation}
which is called the \textit{singular integral}, exists, and that
\begin{equation}
\label{(3.9') is S}
\Big{|} \mu(\infty) - J(L) \Big{|} \ll L^{- 1}.
\end{equation}
We note that $\mu(\infty)$ is the same as what is defined in \cite[(5.3)]{DS}, and we have
\begin{equation}
\label{mu infty positive}
\mu(\infty) > 0
\end{equation}
provided the system of equations
$G_{r}(\mathbf{x}; \mathbf{y}) = 0$ $(1 \leq r \leq R)$
has a non-singular real solution in $(0,1)^{n_1 + n_2}$.

Let us define the following sums:
\begin{equation}
\label{defn Stilde}
\mathcal{S}_{ \mathbf{a}, q } = \sum_{\mathbf{k} \in \mathbb{U}_q^{n_1 + n_2}  } e \left(
\sum_{r=1}^{R}  g_{r} (\mathbf{k}_1; \mathbf{k}_2)  \cdot a_{r} /q \right),
\end{equation}
%$$
%A( q ) = \sum_{ \substack{ \gcd (\mathbf{a},q) = 1 \\  \mathbf{a} \in (\mathbb{Z} / q\mathbb{Z} )^R } } \frac{1}{\phi(q)^{n_1 + n_2} } \ \mathcal{S}_{ \mathbf{a}, q }, % \ e( - \mathbf{s} \cdot \mathbf{m}/q),
%$$
%where $\phi$ is Euler's totient function,
%and
\begin{equation}
\label{def sigular series partial}
A( q ) = \sum_{ \substack{ 0 \leq \mathbf{a} < q \\   \gcd (q, \mathbf{a}) = 1 } } \frac{1}{\phi(q)^{n_1 + n_2} } \ \mathcal{S}_{ \mathbf{a}, q },
\ \ \text{  and  } \ \
\mathfrak{S}(P) = \sum_{q \leq (\log P)^{C_0 R (d_1 + d_2 -1) } }  A( q ).
\end{equation}
\iffalse
%Recall we denote $\mathfrak{B}_0 = [0,1]^n$.
By combining Lemma \ref{Lemma 6 in CM} with the definitions given above we have the following.
\begin{lem}\cite[Lemma 8]{CM}
\label{Lemma 8 in CM} Given any $c>0$, $C > 0$, and $q \leq (\log P)^{C  R (d_1 + d_2 - 1)}$,  we have
$$
\int_{\mathfrak{M}_{\mathbf{a}, q}(C)} S(\boldsymbol{\alpha} ) \ \mathbf{d} \boldsymbol{\alpha}
=
\frac{P_1^{n_1 - d_1 R} P_2^{n_2 - d_2 R } }{\phi(q)^{n_1 + n_2} } \ \mathcal{S}_{\mathbf{a}, q } \  J\left( (\log P)^{C R (d_1 + d_2 - 1)} \right) + O\left(  \frac{ P_1^{n_1 - d_1 R}
 P_2^{n_2 - d_2 R}  } {(\log P)^{c} } \right).
$$
\end{lem}
Therefore, we obtain the following estimate as a consequence of the definition of the major arcs, (\ref{(3.9') is S}), and Lemma \ref{Lemma 8 in CM}.
\fi
Then by combining Lemma \ref{Lemma 6 in CM}, (\ref{(3.9') is S}), and the definition of major arcs, we obtain the following.
\begin{lem}
\label{lemma major arc estimate}
Given any $c>0$,  %and $C_0 > 0$,
we have
\begin{eqnarray}
\label{major arc lem}
\int_{\mathfrak{M}(\vartheta_M) } S(\boldsymbol{\alpha} ) \ \mathbf{d} \boldsymbol{\alpha}
&=&
\mathfrak{S}(P) \mu(\infty) P_1^{n_1 - d_1 R} P_2^{n_2 - d_2 R}
\\
\notag
&+&  O\left( \frac{ P_1^{n_1 - d_1 R} P_2^{n_2 - d_2 R} }{(\log P)^{C_0 R (d_1 + d_2 - 1)}} \sum_{q}  q |A(q)|  +  \frac{P_1^{n_1 - d_1 R} P_2^{n_2 - d_2 R} }{(\log P)^c} \right),
\end{eqnarray}
where the summation in the $O$-term is over $1 \leq q \leq (\log P)^{C_0 R (d_1 + d_2 -1) }$.
%where $P^{\theta_M} = (\log P)^{C_0}$.
\end{lem}

We still have to deal with the term $\mathfrak{S}(P)$, and this is done in the following section.
\subsection{Singular Series}
\label{section singular series}
We now bound $\mathcal{S}_{ \boldsymbol{a}, q }$ when $q$ is a prime power. In order to simplify the exposition let us define
$$
\mathcal{B} := \min\{  \textnormal{codim } V_{\mathbf{G}, 1}^*, \textnormal{codim } V_{\mathbf{G}, 2}^*  \}
%$$
%and
%$$
\ \ \text{  and  } \ \
Q := \frac{1}{2} \cdot \frac{\mathcal{B}}{2^{d_1 + d_2 -2}(R+1) (d_1 + d_2)  }.
$$
Since  $d_1 + d_2 \geq 3$ and (\ref{assmpn on F}) implies $Q > 4R(d_1 + d_2 - 1)/(d_1 + d_2)$,
we can verify that
\begin{equation}
\label{ineq on Q}
Q > \frac{1 + R(2d_1 + 2d_2 + 1)}{2d_1 + 2d_2}
\ \
\text{  and  }
\ \
Q > \frac{R+1}{1 - \frac{1}{2d_1 + 2d_2}} > R+1.
\end{equation}
\begin{lem}
\label{to bound local factor}
Let $p$ be a prime and let $q = p^t$, $t \in \mathbb{N}.$ Let $0 \leq \mathbf{a} < q$ with $\gcd(q, \mathbf{a})=1$.
Let $\varepsilon > 0$ be sufficiently small.
Then %for any $\varepsilon > 0$
we have the following bounds
\begin{eqnarray}
%s \geq
\notag
\mathcal{S}_{ \boldsymbol{a}, q } \ll
\left\{
    \begin{array}{ll}
         p^{-\varepsilon} q^{n_1 + n_2 - Q}
         &\mbox{if } t \leq 2(d_1 + d_2),\\
         p^{Q - \varepsilon} q^{n_1 + n_2 -Q}
         &\mbox{if } t > 2(d_1 + d_2),
    \end{array}
\right.
\end{eqnarray}
where the implicit constants are independent of $p$ and $t$.
\end{lem}

\begin{proof}
We consider the two cases $t \leq 2(d_1 + d_2)$ and $t > 2 ( d_1 + d_2)$ separately.
We begin with the case $t \leq 2(d_1 + d_2)$. In this case
we apply the inclusion-exclusion principle (see \cite[(7.3)]{CM}) %to $\mathcal{S}_{ \mathbf{a}, q }$,
and express $\mathcal{S}_{ \mathbf{a}, q }$ as
%As a result we obtain
\begin{eqnarray}
%\label{bound on S tilde}
%&&
%\\
%\mathcal{S}_{ \mathbf{a}, q }
%&=&  \sum_{\mathbf{k} \in (\mathbb{U}_q)^{n_1 + n_2}} e \left( \sum_{r=1}^{R} g_{r}(\mathbf{k}_1; \mathbf{k}_2)  \cdot a_{r} /q  \right)
%\notag
%\\
%&=&  \sum_{\mathbf{k} \in (\mathbb{Z} / q \mathbb{Z})^{n_1 + n_2}} \prod_{i=1,2} \prod_{j=1}^{n_i} \left( 1 -  \sum_{v_{i,j} \in \mathbb{Z} / p^{t-1} \mathbb{Z} %} \mathbbm{1}_{k_{i,j} = p v_{i,j} }  \right)
%e \left( \sum_{r=1}^{R}  g_{r}(\mathbf{k}_1; \mathbf{k}_2)  \cdot a_{r} /q  \right)
%\notag
%\\
%\notag
%&=&
\label{bound on S tilde}
\sum_{ \substack{I_1 \subseteq \{ 1, 2, \ldots, n_1 \}  \\ I_2 \subseteq \{ 1, 2, \ldots, n_2 \} } } (-1)^{|I_1| + |I_2|}
\sum_{ \mathbf{v} \in  (\mathbb{Z} / p^{t-1} \mathbb{Z} )^{|I_1| + |I_2|} } \
\sum_{\mathbf{k} \in (\mathbb{Z} / q \mathbb{Z})^{n_1 + n_2} }
\mathfrak{H}_{I_1, I_2} (\mathbf{k}; \mathbf{v}) \ e \left( \sum_{r=1}^{R}  g_{r} (\mathbf{k}_1; \mathbf{k}_2)  \cdot a_{r} /q  \right),
\end{eqnarray}
where
\iffalse
$$
\mathbbm{1}_{k_{i,j} = p v_{i,j} } =
\left\{
    \begin{array}{ll}
         1,
         &\mbox{if } k_{i,j} = p v_{i,j} ,\\
         0,
         &\mbox{if } k_{i,j} \not = p v_{i,j},
    \end{array}
\right.
\  \  \
\text{   and   }
\  \  \  \
\mathfrak{H}_{I_1, I_2}(\mathbf{k}; \mathbf{v}) = \prod_{i = 1,2 } \  \prod_{ j \in I_i} \mathbbm{1}_{k_{i,j} = p v_{i,j}}
$$
for $\mathbf{v} \in (\mathbb{Z} / p^{t-1} \mathbb{Z} )^{|I_1| + |I_2|}$.
In other words,
\fi
$\mathfrak{H}_{I_1, I_2} (\mathbf{k}; \mathbf{v})$ is the characteristic function of the set $\{ \mathbf{k} \in (\mathbb{Z} / q \mathbb{Z})^{n_1 + n_2} : k_{i,j} = p v_{i,j} \ (j \in I_i, i= 1,2)\}$.
Here we are using the notations $\mathbf{k} = (\mathbf{k}_1, \mathbf{k}_2)$ where $\mathbf{k}_i \in (\mathbb{Z} / q \mathbb{Z})^{n_i}$, and
$\mathbf{v} = (\mathbf{v}_1, \mathbf{v}_2)$ where $\mathbf{v}_i = (v_{i,j_1}, \ldots, v_{i, j_{|I_i|}})\in (\mathbb{Z} / p^{t-1} \mathbb{Z})^{|I_i|}$
and $I_i = \{ j_1, \ldots, j_{|I_i|} \}$.
We now bound the summand in the expression ~(\ref{bound on S tilde}) by further considering two cases,
$|I_1| + |I_2| > \frac{ \mathcal{B}}{ 2^{d_1 + d_2 -2 } (R+1) }$ and $|I_1| + |I_2| \leq \frac{ \mathcal{B}}{ 2^{d_1 + d_2 -2 } (R+1) }$.
In the first case $|I_1| + |I_2| > \frac{ \mathcal{B}}{ 2^{d_1 + d_2 -2 } (R+1) }$, we use the following trivial estimate
\begin{eqnarray}
&&\Big{|} \sum_{ \mathbf{v} \in  (\mathbb{Z} / p^{t-1} \mathbb{Z})^{|I_1| + |I_2|} } \
\sum_{\mathbf{k} \in (\mathbb{Z} / q \mathbb{Z})^{n_1 + n_2} }
\mathfrak{H}_{I_1, I_2}(\mathbf{k}; \mathbf{v}) \ e \left( \sum_{r=1}^{R}  g_{r}(\mathbf{k}_1; \mathbf{k}_2)  \cdot a_{r} /q  \right) \Big{|}
\label{est 1}
\\
&\leq&
p^{(t-1) (|I_1| + |I_2|) } (p^t)^{n_1 + n_2  - |I_1| - |I_2|}
\notag
\\
%&=& q^{n_1 + n_2 - (|I_1| + |I_2|)/t}
%\notag
%\\
%&\leq& q^{n_1 + n_2  - \frac{ \mathcal{B}}{ 2^{d_1 + d_2 -2 } (R+1) (2 d_1 + 2 d_2) }  }
%\notag
%\\
%&=&
&\leq&
\notag
%\label{est 1}
q^{n_1 + n_2  - Q - \varepsilon}.
\end{eqnarray}

On the other hand, suppose  $|I_1| + |I_2| \leq \frac{ \mathcal{B}}{ 2^{d_1 + d_2 -2 } (R+1) }$. Let us label $\mathbf{s} = (s_{1}, \ldots , s_{n_1 - |I_1|})$ and $\mathbf{w}= (w_{1}, \ldots , w_{n_2 - |I_2|} )$ to be the remaining variables of $\mathbf{x}$ and $\mathbf{y}$ after setting $x_j = 0$ for each $j \in I_1$ and $y_{j'} = 0$ for each $j' \in I_2$ respectively. For each $1 \leq r \leq R$, %let
%$$
%\mathfrak{f}_{r}(\mathbf{s}; \mathbf{w}) = g_{r}(\mathbf{x}; \mathbf{y}) |_{x_j = p v_{1,j} (j \in I_1), y_{j'} = p v_{2,j'} ( j' \in I_2 ) },
%$$
%or equivalently the polynomial
let $\mathfrak{f}_{r}(\mathbf{s}; \mathbf{w})$ be the polynomial obtained by
substituting $x_j = p v_{1,j}$ $(j \in I_1)$ and $y_{j'} = p v_{2,j'}$ $(j' \in I_2)$ to the polynomial ${g}_{r}(\mathbf{x}; \mathbf{y})$.
Thus $\mathfrak{f}_{r} (\mathbf{s};\mathbf{w})$ is a polynomial in $\mathbf{s}$ and $\mathbf{w}$ whose coefficients may depend on $p$ and $\mathbf{v}$.
With these notations we have
\begin{eqnarray}
\label{ineq 3}
\sum_{\mathbf{k} \in (\mathbb{Z} / q \mathbb{Z})^{n_1 + n_2}} \mathfrak{H}_{I_1, I_2}(\mathbf{k}; \mathbf{v}) \ e \left( \sum_{r=1}^{R}  g_{r}(\mathbf{k}_1; \mathbf{k}_2)  \cdot a_{r} /q  \right)
%&=&
%\sum_{ \substack{ \mathbf{s} \in (\mathbb{Z} / q \mathbb{Z})^{n_1 - |I_1|} \\   \mathbf{w} \in (\mathbb{Z} / q \mathbb{Z})^{n_2 - |I_2|}   } } e \left(  \sum_{r=1}^{R}  \mathfrak{f}_{r} (\mathbf{s}; \mathbf{w})  \cdot a_{r} /q  \right)
%\\
%\notag
&=&
\sum_{ \substack{ \mathbf{s} \in [0, q-1]^{n_1 - |I_1|} \\  \mathbf{w} \in [0, q-1]^{n_2 - |I_2|}} } e \left(  \sum_{r=1}^{R}  \mathfrak{f}_{r} (\mathbf{s}; \mathbf{w})  \cdot a_{r} /q  \right).
\end{eqnarray}
We can also deduce easily that the homogeneous degree $(d_1 + d_2)$ portion of the polynomial $\mathfrak{f}_{r}(\mathbf{s}; \mathbf{w})$, which we denote
$\mathfrak{F}_{r}(\mathbf{s};\mathbf{w})$, is obtained by substituting $x_j = 0 \ (j \in I_1)$ and $y_{j'} = 0 \ (j' \in I_2)$ to ${G}_{r}(\mathbf{x};\mathbf{y})$.
%Hence, we have
%$$\mathfrak{F}_{r}(\mathbf{s};\mathbf{w}) =  {G}_{r}(\mathbf{x};\mathbf{y}) |_{x_j = 0 \ (j \in I_1), y_{j'} = 0 \ (j' \in I_2)},
%$$
%and i
In particular, it is independent of $p$ and $\mathbf{v}$.
It then follows from Lemma \ref{Lemma on the B rank} that
\begin{eqnarray}
\min \{ \textnormal{codim}(  V^*_{ \boldsymbol{\mathfrak{F}}, 1} ), \textnormal{codim}(  V^*_{ \boldsymbol{\mathfrak{F}}, 2} ) \}
&\geq& \mathcal{B} - (R+1)(|I_1| + |I_2|)
%\notag
%\\
%&>& \mathcal{B} - \frac{ \mathcal{B}}{ 2^{d_1 + d_2 -2 }}
%\notag
%\\
%&=&
\notag
\geq
\left( 1 -  \frac{1}{ 2^{d_1 + d_2 -2 }} \right) \mathcal{B}.
\end{eqnarray}

Let $\varepsilon' > 0$ be sufficiently small. Thus by Corollary \ref{cor 1 of Lemma 4.3 in DS} we obtain
\begin{eqnarray}
\notag
\sum_{ \substack{ \mathbf{s} \in [0, q-1]^{n_1 - |I_1|} \\  \mathbf{w} \in [0, q-1]^{n_2 - |I_2|} } } e \left( \sum_{r=1}^{R} \mathfrak{f}_{r} (\mathbf{s};\mathbf{w})  \cdot a_{r} /q  \right)
&\ll& q^{n_1 + n_2 - |I_1| - |I_2| -
%\left( \left( 1 -  \frac{1}{ 2^{d_1 + d_2 -2 }} \right) \mathcal{B} - \delta_0 \right)
%\frac{1}{2^{d_1 + d_2 -2}  R (d_1 + d_2 -1)  }
\frac{  \left( 1 -  2^{-d_1 - d_2 +2 } \right) \mathcal{B} - \delta_0  }
{2^{d_1 + d_2 -2}  R (d_1 + d_2 -1)  }
+ \varepsilon' }
\\
\notag
&\leq& q^{n_1 + n_2 - |I_1| - |I_2| - Q - \varepsilon}.
\end{eqnarray}
%Note we used the fact that at least one of $d_1$ or $d_2$ is strictly greater than $1$ to obtain the final inequality above.
Consequently, we have from (\ref{ineq 3}) that
\begin{eqnarray}
&&\Big{|} \sum_{ \mathbf{v} \in  (\mathbb{Z} / p^{t-1} \mathbb{Z})^{|I_1| + |I_2|} } \
\sum_{\mathbf{k} \in (\mathbb{Z} / q \mathbb{Z})^{n_1 + n_2} }
\mathfrak{H}_{I_1, I_2}(\mathbf{k}; \mathbf{v}) \ e \left( \sum_{r=1}^{R}  g_{r}(\mathbf{k}_1; \mathbf{k}_2)  \cdot a_{r} /q  \right) \Big{|}
\label{est 2}
\\
&\leq&
p^{(t-1) (|I_1| + |I_2|) } q^{n_1 + n_2 - |I_1| - |I_2| - Q - \varepsilon}
\notag
\\
&\leq&
\notag
q^{n_1 + n_2  - Q - \varepsilon}
\end{eqnarray}
in this case as well.
By applying the estimates (\ref{est 1}) and (\ref{est 2}) in (\ref{bound on S tilde}), we obtain
%$$
%\mathcal{S}_{ \mathbf{a}, q} \ll %q^{n_1 + n_2 - \frac{\mathcal{B}}{2^{d_1 + d_2} (R+1) (d_1 + d_2 + 1)} + \varepsilon} =
%q^{n_1 + n_2 - Q}
%$$
the desired estimate for the case $t \leq 2(d_1 + d_2)$.

We now consider the case $t > 2(d_1 + d_2)$. By the definition of $\mathcal{S}_{ \mathbf{a}, q }$ we have
\begin{eqnarray}
\label{widetile S part 3-1}
\mathcal{S}_{ \mathbf{a}, q }
=
%&=&  \sum_{\mathbf{k} \in (\mathbb{U}_q)^{n_1 + n_2 }} e \left( \sum_{r=1}^{R}  g_{r} (\mathbf{k}_1; \mathbf{k}_2)  \cdot a_{r} /q  \right)
%\\
%&=& \sum_{\mathbf{k} \in (\mathbb{U}_p)^{n_1 + n_2}} \  \sum_{ \mathbf{b} \in (\mathbb{Z} / p^{t-1} \mathbb{Z} )^{n_1 + n_2}  }
%e \left( \sum_{r=1}^{R}  g_{r} (\mathbf{k}_1 + p \mathbf{b}_1; \mathbf{k}_1 + p \mathbf{b}_1 )  \cdot a_{r} /q  \right)
%\notag
%\\
%&=&
\sum_{\mathbf{k} \in \mathbb{U}_p^{n_1 + n_2}} \  \sum_{ \substack{  \mathbf{b}_1 \in [0,p^{t-1}-1]^{n_1}  \\  \mathbf{b}_2 \in [0,p^{t-1}-1]^{n_2}  }  } e \left( \sum_{r=1}^{R}  g_{r} (\mathbf{k}_1 + p \mathbf{b}_1 ; \mathbf{k}_2 + p \mathbf{b}_2 )  \cdot a_{r} /q  \right).
\end{eqnarray}
For each fixed $\mathbf{k} \in \mathbb{U}_p^{n_1 + n_2}$, we have
$$
g_{r} (\mathbf{k}_1 + p \mathbf{b}_1; \mathbf{k}_2 + p \mathbf{b}_2 )  = p^{d_1 + d_2}  G_{r} ( \mathbf{b}_1; \mathbf{b}_2)
+ \varpi_{r; p, \mathbf{k}}(\mathbf{b}) \ \ \ (1 \leq r \leq R),
$$
where $\varpi_{r; p, \mathbf{k}}(\mathbf{b})$ is a polynomial in $\mathbf{b} = (\mathbf{b}_1, \mathbf{b}_2)$ of degree at most $d_1 + d_2 - 1$. Clearly every monomial of  $\varpi_{r; p, \mathbf{k}}(\mathbf{b})$  has degree in $\mathbf{b}_i$ strictly less than $d_i$ for one of $i=1$ or $2$, and its coefficients are integers which may depend on $p$ and $\mathbf{k}$.
We let
$$
\mathfrak{c}_r(\mathbf{b}_1; \mathbf{b}_2) = {G}_{r} ( \mathbf{b}_1; \mathbf{b}_2) + \frac{1}{p^{d_1 + d_2}} \
\varpi_{r; p, \mathbf{k}} (\mathbf{b}) \ \ \ (1 \leq r \leq R).
$$
We can then express the inner sum on the right hand side of (\ref{widetile S part 3-1}) as
\begin{equation}
\label{**}
\sum_{ \mathbf{b} \in [0,p^{t-1}-1]^{n_1 + n_2}   } e \left( \sum_{r=1}^{R}  \mathfrak{c}_{r} (\mathbf{b}_1; \mathbf{b}_2 )  \cdot \frac{a_{r}}{q/p^{d_1 + d_2}}  \right).
\end{equation}
We have that each $\mathfrak{c}_r$ has coefficients in $\mathbb{Q}$, and its degree $(d_1 + d_2)$ homogeneous portion $G_r$ has coefficients in $\mathbb{Z}$. We apply Lemma \ref{Lemma 4.3 in DS} with
$\mathfrak{B}_1 = [0,1)^{n_1}$, $\mathfrak{B}_2 = [0,1)^{n_2}$, $\alpha_{r} = a_{r} /p^{t-d_1 - d_2}$ $(1 \leq r \leq R)$, $P_1 = P_2 = p^{t-1}$, and $P = p^{(t-1)(d_1 + d_2)}$. Let $\theta = \frac{1}{2(d_1 + d_2)(d_1 + d_2 - 1) (R+1)} < \frac{1}{d_1 + d_2}$. Suppose there exist $\widetilde{a}_1, \ldots, \widetilde{a}_R$ and
$1 \leq \widetilde{q} \leq P^{R(d_1 + d_2 - 1) \theta}$
such that $\gcd(\widetilde{q}, \widetilde{a}_1, \ldots, \widetilde{a}_R) = 1$ and
$$
2 |\widetilde{q} \alpha_r - \widetilde{a}_r| \leq P_1^{-d_1} P_2^{- d_2} P^{R (d_1 + d_2 - 1) \theta} \ \ (1 \leq r \leq R).
$$
Note from $t + 1 > 2(d_1 + d_2)$ it follows that $(t - d_1 - d_2) > \frac{t - 1}{2}$.
Then it is not possible that $p^{t-d_1 - d_2}$ divide $\widetilde{q}$, because
$$
1 \leq \widetilde{q} \leq P^{R(d_1 + d_2 - 1) \theta} < P^{\frac{1}{2(d_1 + d_2)}} = p^{ \frac{t-1}{2}} < p^{t - d_1 - d_2}.
$$
Since $\gcd(q, a_1, \ldots, a_R) = 1$ and $q = p^t$,  without loss of generality we assume $\gcd(a_1, p) = 1$.
Then $\widetilde{q} \alpha_1$ is not an integer. Thus we have
$$
\frac{1}{p^{t- d_1 - d_2}} \leq |\widetilde{q} \alpha_1 - \widetilde{a}_1| < \frac{1}{2} P_1^{-d_1} P_2^{- d_2} p^{\frac{t-1}{2}} \leq  \frac{1}{p^{(t-1)(d_1 + d_2 - 1/2) } }
$$
which is a contradiction, because $t- d_1 -d_2 < (t - 1) (d_1 + d_2 - 1/2)$.
Therefore, we are in the alternative (ii) of Lemma \ref{Lemma 4.3 in DS}, and the expression (\ref{**}) is bounded by
\begin{eqnarray}
%\left| \sum_{ \mathbf{b} \in [0, P]^{n_1 + n_2}   } e \left( \sum_{r=1}^{R}  \mathfrak{c}_{r} (\mathbf{b}_1; \mathbf{b}_2 )  \cdot \alpha_{r}  \right) \right|
\label{***}
\ll
P_1^{n_1} P_2^{n_2} P^{- \theta \cdot \frac{\mathcal{B} - \delta_0 }{2^{d_1 + d_2 - 2}}  } (\log P)^{n_1}
%\\&=&
\ll (p^{t-1})^{n_1 + n_2 -  \frac{\mathcal{B} - \delta_0 }{ 2 (d_1 + d_2 - 1) (R+1) 2^{d_1 + d_2 - 2} } + \varepsilon' }
%\notag
%\\
%&=&
\leq (p^{t-1})^{n_1 + n_2 - Q - \varepsilon}.
\end{eqnarray}
Thus we can bound ~(\ref{widetile S part 3-1}) by (\ref{**}) and (\ref{***}) as follows
\begin{eqnarray}
| \mathcal{S}_{ \mathbf{a}, q } |
\ll
p^{n_1 + n_2} \ (p^{t-1})^{n_1 + n_2 - Q - \varepsilon}
\leq
p^{Q - \varepsilon} q^{n_1 + n_2  - Q }.
\notag
\end{eqnarray}

\end{proof}

By a similar argument as in \cite[Chapter VIII, \S 2, Lemma 8.1]{H}, one can show that $A(q)$
is a multiplicative function of $q$. We omit the proof of the following lemma as it is a basic exercise
involving the Chinese remainder theorem and manipulating summations.
\begin{lem}
Suppose $q,q' \in \mathbb{N}$ and $\gcd (q,q') = 1$. Then we have
$
A(q q') = A(q) A(q').
$
\end{lem}

Recall we defined the term $\mathfrak{S}(P)$ in (\ref{def sigular series partial}).
For each prime $p$, we define
\begin{equation}
\label{def mu p}
\mu(p) =  1  + \sum_{t=1}^{\infty} A(p^t),
\end{equation}
which converges absolutely under our assumptions on $\mathbf{g}$.
Furthermore, %under our assumptions on $\mathbf{g}$
the following limit exists
\begin{equation}
\label{def sigular series}
\mathfrak{S}(\infty) := \lim_{L \rightarrow \infty} \ \sum_{q \leq L}  A(q)  = \prod_{p \ \text{prime}} \mu(p),
\end{equation}
which is called the \textit{singular series}.
We prove these statements in the following Lemma \ref{singular series lemma}.

\begin{lem}
\label{singular series lemma}
There exists $\delta_1 > 0$ such that for each prime $p$, we have
$
\mu(p) = 1 + O(p^{-1 - \delta_1})
$
where the implicit constant is independent of $p$.
Furthermore, we have
$$
\Big{|} \mathfrak{S}(P) -  \mathfrak{S}(\infty) \Big{|} \ll (\log P)^{-C_0 R (d_1 + d_2 - 1) \delta_2 }
$$
for some $\delta_2 > 0$.
\end{lem}
Therefore, the limit in ~(\ref{def sigular series}) exists, and the product in ~(\ref{def sigular series}) converges.
We leave the details  that these two quantities are equal to the reader.
%it is well defined.
\begin{proof}
For any $t \in \mathbb{N}$, we know that $\phi(p^t) = p^t(1 - 1/p) \geq \frac12 p^t$.
Therefore, by considering the two cases as in the statement of Lemma \ref{to bound local factor} we obtain
\begin{eqnarray}
| \mu(p) - 1 |
%&\leq&
%\sum_{1 \leq t \leq 2(d_1 + d_2)} \Big{|} \sum_{ \substack{ \gcd (\mathbf{a}, p^t) = 1 \\  \mathbf{a} \in (\mathbb{Z} / p^t \mathbb{Z} )^R } } \frac{1}{\phi(p^t)^{n_1 + n_2}} \  \mathcal{S}_{ \mathbf{a}, p^t }  \Big{|}
%+
%\sum_{ t > 2 (d_1 + d_2)} \Big{|} \sum_{ \substack{ \gcd (\mathbf{a}, p^t) = 1 \\  \mathbf{a} \in (\mathbb{Z} / p^t \mathbb{Z} )^R } }   \frac{1}{\phi(p^t)^{n_1 + n_2} } \  \mathcal{S}_{ \mathbf{a}, p^t } \Big{|}
%\notag
%\\
&\ll&
\sum_{1 \leq t \leq 2(d_1 + d_2)} p^{tR} p^{-(n_1 + n_2)t} p^{(n_1 + n_2)t - t Q}
+
\sum_{ t > 2(d_1 + d_2)} p^{tR} p^{-(n_1 + n_2)t} p^{Q + (n_1 + n_2)t - t Q }
\notag
\\
&\ll&
p^{R - Q}
+
p^Q p^{-(2d_1 + 2d_2 + 1)(Q-R)}
\notag
\\
&\ll&
p^{-1 - \delta_1}
\notag
\end{eqnarray}
for some $\delta_1 > 0$, where the last inequality follows from (\ref{ineq on Q}). We note that the implicit constants in $\ll$ are independent of $p$ here.

Let $q = p_1^{t_1} \cdots p_{v}^{t_{v}}$ be the prime factorization of $q \in \mathbb{N}$.
Without loss of generality, suppose we have $t_j \leq 2( d_1 + d_2) \ (1 \leq j \leq v_0)$ and
$t_j > 2( d_1 + d_2) \ (v_0 < j \leq v)$. Note we can assume the implicit constant in Lemma \ref{to bound local factor}
is $1$ for $p$ sufficiently large with the cost of $p^{-\varepsilon}$.
By a similar calculation as above
and the multiplicativity of $A(\cdot)$, it follows that
\begin{eqnarray}
\label{bound on Aq}
A(q)
%&=&
%A(p_1^{t_1}) \ \ldots \  A(p_{v}^{t_{v}})
%\notag
%\\
%&\ll&
%\left( \prod_{j=1}^{v_0} p_j^{t_j R} p_j^{-(n_1 + n_2) t_j}  p_j^{t_j (n_1 + n_2 - Q )} \right) \cdot \left(  \prod_{j=v_0 + 1}^{v}  p_j^{t_j R} p_j^{-(n_1 + n_2) t_j} p_j^{Q} p_j^{t_j (n_1 + n_2 - Q + \varepsilon)} \right)
%\notag
%\\
%&\leq&
\ll
q^{R - Q } \cdot  \left(  \prod_{j=v_0 + 1}^{v}  p_j^{Q} \right)
%\\
%&\leq&
\leq
q^{R - Q } \cdot  q^{\frac{Q}{2d_1 + 2d_2} }
%\notag
%\\
%&\leq&
\leq
q^{- 1 - \delta_2}
\end{eqnarray}
for some $\delta_2 > 0$, where we obtained the last inequality from (\ref{ineq on Q}). We note that the implicit constant in $\ll$ is independent of $q$ here.
%We know that $n \geq h(f_b)$.
Therefore, we obtain
\begin{eqnarray}
\Big{|} \mathfrak{S}(P) -  \mathfrak{S}(\infty) \Big{|}
\leq
 \sum_{q > (\log P)^{C_0 (d_1 + d_2 - 1) R  } } |  A(q) |
\notag
%\\
%&\ll&
%\sum_{q > (\log L)^{C_0 (d_1 + d_2 - 1) R  } } q^{- 1 - \delta_2}
%\notag
%\\
%&\ll&
\ll
(\log P)^{- C_0 (d_1 + d_2 - 1) R  \delta_2 }.
\notag
\end{eqnarray}
\end{proof}

Using the bound (\ref{bound on Aq}), we obtain that the first term in the $O$-term of (\ref{major arc lem}) is bounded by
\begin{eqnarray}
\label{error bound}
&&\frac{P_1^{n_1} P_2^{n_2}}{P^R (\log P)^{C_0 R(d_1 + d_2 -1)}} \sum_{1 \leq q \leq (\log P)^{C_0 R(d_1 + d_2 -1)} }  q |A(q)|
\\
&\ll&
\frac{P_1^{n_1} P_2^{n_2}}{P^R (\log P)^{C_0 R(d_1 + d_2 -1)}} \sum_{1 \leq q \leq (\log P)^{C_0 R(d_1 + d_2 -1)} }  q^{- \delta_2}
\notag
\\
&\ll&
\frac{P_1^{n_1} P_2^{n_2}}{P^R (\log P)^{C_0 R(d_1 + d_2 -1)}}(\log P)^{C_0 R(d_1 + d_2 -1)(1- \delta_2) }
\notag
\\
&\ll&
\frac{P_1^{n_1} P_2^{n_2}}{P^R }(\log P)^{- C_0 R(d_1 + d_2 -1) \delta_2 }.
\notag
\end{eqnarray}

Let $\nu_t(p)$ denote the number of solutions $(\mathbf{x}, \mathbf{y}) \in (\mathbb{U}_{p^t})^{n_1 + n_2}$
to the congruence relations $g_{r}( \mathbf{x} ; \mathbf{y} ) \equiv 0 \  (\text{mod } p^t)$ $(1 \leq r \leq R).$
\iffalse
Then using the fact that
\begin{eqnarray}
%s \geq
\notag
\sum_{ a \in \mathbb{Z}/ p^t \mathbb{Z} } e \left(  m_0  \cdot {a}/p^t  \right) =
\left\{
    \begin{array}{ll}
         p^{t}
         &\mbox{if } p^t | m_0 ,\\
         0
         &\mbox{otherwise,}
    \end{array}
\right.
\end{eqnarray}
\fi
%we can deduce
It is then a basic exercise (see \cite[pp. 58]{Y}) to deduce
\begin{eqnarray}
1 + \sum_{j=1}^t A(p^j)
=
\frac{p^{tR}}{\phi(p^t)^{n_1 + n_2}  } \ \nu_t(p).
\notag
\end{eqnarray}
Therefore, under our assumptions on $\mathbf{g}$ we obtain
$$
\mu(p) = \lim_{t \rightarrow \infty}  \frac{  p^{tR} \ \nu_t(p) }{ \phi(p^t)^{n_1 + n_2}  }.
$$
We can then deduce by an application of  Hensel's lemma
that $\mu(p) > 0$,
if the system ~(\ref{set of eqn 3}) has a non-singular solution in $(\mathbb{Z}_p^{\times})^{n_1 + n_2}$. From this it follows in combination with (\ref{def sigular series}) and Lemma \ref{singular series lemma} that if the system ~(\ref{set of eqn 3}) has a non-singular solution in $(\mathbb{Z}_p^{\times})^{n_1 + n_2}$ for every prime $p$, then
\begin{equation}
\label{mu p positive}
\mathfrak{S}(\infty) = \prod_{p \ \text{prime}}\mu(p) > 0.
\end{equation}
By combining (\ref{error bound}) and Lemmas \ref{lemma major arc estimate} and \ref{singular series lemma}, we obtain the following.
\begin{prop}
\label{prop major}
%Let $\mathbf{f}$ be the polynomials in (\ref{set of eqn 3}).
Given any $c>0$, under our assumptions on $\mathbf{g}$ the following holds
$$
\int_{\mathfrak{M}(\theta_M)} S(\boldsymbol{\alpha} ) \ \mathbf{d} \boldsymbol{\alpha}
=  \mathfrak{S}(\infty) \mu(\infty) \  P_1^{n_1 - R d_1} P_2^{n_2 - R d_2} + O\left( \frac{ P_1^{n_1 - R d_1} P_2^{n_2 - R d_2} }{(\log P)^{c}} \right),
$$
where $P^{\theta_M} = (\log P)^{C_0}$.
\end{prop}

Finally, it is clear that Theorem \ref{thm main bihmg} follows from (\ref{orthog reln}) and Propositions \ref{prop minor} and \ref{prop major}. The fact that
under suitable local conditions, $\sigma_{\mathbf{g}} = \mathfrak{S}(\infty) \mu(\infty) > 0$ follows from (\ref{mu infty positive}) and (\ref{mu p positive}).

\section{Proof of Theorem \ref{thm main1} }
\label{sec final}
We begin this section by proving the following theorem.
\begin{thm}
\label{mainineqthm}
Let $d > 1$. Let $F(\mathbf{x}) \in \mathbb{Z}[x_1, \ldots, x_n]$ be a degree $d$ homogeneous form. Let us define
a bihomogeneous form
$$
G(\mathbf{x}; \mathbf{y}) = F(x_1y_1, \ldots, x_ny_n).
$$
Then we have
$$
\min \{ \textnormal{codim } V_{{G}, 1}^*, \textnormal{codim } V_{{G}, 2}^*  \} \geq  \frac{\textnormal{codim } V_{{F}}^*}{2}.
$$
\end{thm}

\begin{proof}
Let $X$ be an irreducible component of $V_{G,1}^*$ such that $\dim X = \dim V_{G,1}^*$.
By relabeling the variables if necessary, let us suppose we have
$$
X \not \subseteq V(y_j) \ \ (1 \leq j \leq m) \ \ \textnormal{  and  } \  \  X \subseteq V(y_i) \ \ (m + 1 \leq j \leq n)
$$
for some $0 \leq m \leq n$.

Claim 1:  There exists $(z_1, \ldots, z_m) \in (\mathbb{C}\backslash \{ 0\})^m$ such that
$$
\dim X \cap ( \cap_{1 \leq j \leq m} V(y_j - z_j) ) \geq \dim X - m.
$$
\begin{proof}[Proof of Claim 1]
First we show that there exists $(z_1, \ldots, z_m) \in (\mathbb{C}\backslash \{ 0\})^m$
such that $X \cap ( \cap_{1 \leq j \leq m} V(y_j - z_j) ) \not = \emptyset$.
Suppose such $(z_1, \ldots, z_m)$ does not exist. Then we have $X = \cup _{1 \leq j \leq m} X \cap V(y_j)$.
Since $X$ is irreducible, this implies $X = X \cap V(y_{j_0})$ for some $1 \leq j_0 \leq m$; we have a contradiction because
$X \not \subseteq V(y_{j_0})$.

Let $P = (\mathbf{x}_0, z_1, \ldots, z_m, \mathbf{0}) \in X$ with $(z_1, \ldots, z_m) \in (\mathbb{C}\backslash \{ 0\})^m$.
Let us consider
$$
\emptyset  \not =   X \cap V(y_1 - z_1) =  \cup_{1 \leq j \leq \ell_1} W_{1, j},
$$
where $W_{1, j}$'s are the irreducible components of $X \cap V(y_1 - z_1)$.
Recall if $Z$ is an irreducible affine variety and $H$ is a hypersurface, then
we have one of: $Z \cap H = Z$, $Z \cap H = \emptyset$ and every irreducible component of $Z \cap H$ has dimension $\dim Z - 1$.
Therefore, it follows that the $\dim W_{1, j} \geq \dim X - 1$ for each $1 \leq j \leq \ell_1$.

Next without loss of generality suppose $P \in W_{1, 1}$.
Let us consider
$$
\emptyset  \not = W_{1,1} \cap V(y_2 - z_2) = \cup_{1 \leq j \leq \ell_2} W_{2, j},
$$
where $W_{2, j}$'s are the irreducible components of $W_{1,1} \cap V(y_2 - z_2)$. By the same argument as above, we obtain
$$
\dim W_{2, j} \geq  \dim W_{1, 1} -  1   \geq  \dim X - 2 \ \ (1 \leq j \leq \ell_2).
$$
By continuing in this manner, we obtain the result.
\end{proof}
Let us fix $(z_1, \ldots, z_m) \in (\mathbb{C}\backslash \{ 0\})^m$ as in Claim 1.
Let $z_{m+1} = \cdots = z_n = 0$. Then we have
\begin{eqnarray}
\label{1.2coorect'}
\dim X \cap ( \cap_{1 \leq j \leq n} V(y_j - z_j) ) &=& \dim X \cap ( \cap_{1 \leq j \leq m} V(y_j - z_j) )
\\
\notag
&\geq& \dim X - m
\\
\notag
&=& \dim V_{G,1}^* - m.
\end{eqnarray}
We also have
\begin{eqnarray}
\label{1.2coorect}
&&X \cap ( \cap_{1 \leq j \leq n} V(y_j - z_j) )
\\
\notag
&\subseteq&
V_{G,1}^* \cap ( \cap_{1 \leq j \leq n} V(y_j - z_j) )
\\
\notag
&=&
\Big{\{} \mathbf{x} \in \mathbb{C}^{n} : \frac{\partial F}{ \partial x_1} (x_1z_1, \ldots, x_m z_m, \mathbf{0} )
= \cdots = \frac{\partial F}{ \partial x_m} (x_1z_1, \ldots, x_m z_m, \mathbf{0} )  = 0 \Big{\}}
\\
\notag
&\times& \{ \mathbf{y} \in \mathbb{C}^n : y_j = z_j \ (1 \leq j \leq n)  \}.
\end{eqnarray}
For each $1 \leq k \leq n$, let us define
$$
T_{k} = \Big{\{} \mathbf{x} \in \mathbb{C}^{n} : \frac{\partial F}{ \partial x_1} (\mathbf{x})
= \cdots = \frac{\partial F}{ \partial x_k} (\mathbf{x})  = x_{k+1} = \cdots = x_n = 0 \Big{\}}.
$$
Then it follows from (\ref{1.2coorect}) that
\begin{eqnarray}
\label{1.3coorect}
\dim X \cap ( \cap_{1 \leq j \leq n} V(y_j - z_j) ) \leq (n-m) + \dim T_m.
\end{eqnarray}

Claim 2: We have
\begin{eqnarray}
\label{boundonTm}
\max_{1 \leq k \leq n} \dim T_{k} \leq \frac{n + \dim V_F^*}{2}.
\end{eqnarray}
\begin{proof}[Proof of Claim 2]
First we have
$$
\dim T_{k+1} -1 \leq \dim T_{k} \leq  \dim T_{k+1} + 1.
$$
This is because the dimension of
$$
\Big{\{} \mathbf{x} \in \mathbb{C}^{n} : \frac{\partial F}{ \partial x_1} (\mathbf{x})
= \cdots = \frac{\partial F}{ \partial x_k} (\mathbf{x})  = x_{k+2} = \cdots = x_n = 0 \Big{\}}
$$
is either $\dim T_{k+1}$ or $\dim T_{k+1} + 1$. Furthermore,
intersecting this set with $V(x_{k+1})$, which is $T_k$, either reduces the dimension by $1$
or the dimension stays the same. Therefore, we have $\dim T_{k+1} -1 \leq \dim T_k \leq \dim T_{k+1} + 1$.
Here it is important that we are only dealing with homogeneous forms, because every irreducible component of
an affine variety $Z$ defined by homogeneous forms contains $\mathbf{0}$; therefore, any hypersurface $H$
defined by a homogeneous form intersects every irreducible component of $Z$, and thus we always have $\dim Z \cap H \geq \dim Z - 1$
in this case.

Let $L_1, \ldots, L_n$ be a set of integers satisfying $L_n = \dim V_F^*$, $0 \leq L_k \leq k$ $(1 \leq k \leq n)$ and
$$
L_{k+1} -1 \leq L_{k} \leq  L_{k+1} + 1 \ \ (1 \leq k \leq n-1).
$$
Then it is a basic exercise to show that the largest possible value of $\max_{1 \leq k \leq n}L_k$ for any such set of integers is $k_0$, where
$$
k_0 = \left\{
    \begin{array}{ll}
         \frac{n + \dim V_F^*}{2}
         &\mbox{if } n \equiv \dim V_F^* \ (\textnormal{mod }2),\\
         \frac{n + \dim V_F^*}{2} - \frac12
         &\mbox{if } n \not \equiv \dim V_F^* \ (\textnormal{mod }2).
    \end{array}
\right.
$$
Since we can choose $L_k = \dim T_k$ $(1 \leq k \leq n)$, the result follows.
\end{proof}
Therefore, by combining (\ref{1.2coorect'}), (\ref{1.3coorect}) and (\ref{boundonTm}), we obtain
$$
\textnormal{codim }V_{G,1}^* = 2n - \dim V_{G,1}^* \geq  \frac{n - \dim V_F^*}{2}  = \frac{\textnormal{codim }V_F^*}{2}.
$$
By symmetry we obtain the same bound for $\textnormal{codim }V_{G,2}^*$ as well.
\end{proof}

Let $d > 1$. Throughout this section we let
$f(\mathbf{x})$ be a degree $d$ polynomial in $\mathbb{Z}[x_1, \ldots, x_n]$, and denote its degree $d$ homogeneous portion by $F(\mathbf{x})$.
We now solve the equation
\begin{equation}
\label{eqn main 2}
f(\mathbf{x}) = 0
\end{equation}
in semiprimes.

Let $N = N_1 N_2$ where $N_1 \geq N_2$. Let us define
\begin{equation}
\mathcal{N}_{2}(f; N ; N_1, N_2 ) =
\sum_{\substack{ z_1 \in [0, N] \\ z_1 = p_1 q_1, \ p_1 \geq q_1 \\ p_{1} \in [0, N_1] \cap \wp \\ q_{1} \in [0, N_2] \cap \wp } }
\cdots
\sum_{\substack{ z_n \in [0, N] \\ z_n = p_n q_n, \ p_n \geq q_n \\ p_{n} \in [0, N_1] \cap \wp \\ q_{n} \in [0, N_2] \cap \wp } }
\prod_{j=1}^n (\log p_{j}) (\log q_j)  \  \cdot  \ \mathbbm{1}_{ V(f) } (\mathbf{z}).
\end{equation}
It is clear that $\mathcal{N}_{2}(f; N ; N_1, N_2 )$ is the number of semiprime solutions $(p_1 q_1, \ldots, p_n q_n) \in [0,N]^n$ to the equation (\ref{eqn main 2}), where $p_j \geq q_j$, $p_j \in [0, N_1] \cap \wp$, and  $q_j \in [0, N_2] \cap \wp$, counted with weight
$\prod_{1 \leq j \leq n} (\log p_j) (\log q_j)$. We also consider the following modification of the local conditions ($\star$) given in Section \ref{sec intro}.
\medskip

\noindent \textbf{Local conditions ($\star'$).} The equation
\begin{equation}
\label{local real}
F(\mathbf{x}) = 0
\end{equation}
has a non-singular real solution in $(0,1)^{n}$, and the equation (\ref{eqn main 2})
has a non-singular solution in $( \mathbb{Z}_p^{\times})^n$ for every prime $p$.
\medskip

It is clear that these conditions are identical to the local conditions ($\star$) %given in Section \ref{sec intro} 
when the polynomial in consideration is homogeneous. We prove the following theorem.
\begin{thm}
\label{thm main 2}
Let $\delta \leq 1/2$. Suppose that $f$ satisfies the local conditions \textnormal{($\star'$)} and
$$
\textnormal{codim } V_{F}^* >  2 \cdot  4^d  \max \left\{  4 (2d-1), \frac{d}{\delta}  \right\}.
$$
Then we have
\begin{equation}
\notag
\mathcal{N}_2(f; N; N^{1 - \delta}, N^{\delta} ) \gg N^{n - d}.
\end{equation}
\end{thm}
By taking $\delta = 1/2$ in the above theorem, the following is an immediate corollary which also implies Theorem \ref{thm main1}.
\begin{cor}
\label{cor main}
Suppose that $f$ satisfies the local conditions \textnormal{($\star'$)} and
$
\textnormal{codim } V_{F}^* > 4^d \cdot 8(2d-1).
$
Then we have
$
\mathcal{N}_2(f; N ; \sqrt{N}, \sqrt{N}) \gg N^{n - d}.
$
\end{cor}
\begin{proof}[Proof of Theorem \ref{thm main 2}]
We define $g(\mathbf{x}; \mathbf{y}) = f(x_1 y_1, \ldots, x_n y_n)$,
and denote its degree $2d$ homogeneous portion by $G(\mathbf{x}; \mathbf{y}) = F(x_1 y_1, \ldots, x_n y_n)$,
which is bihomogeneous in $\mathbf{x}$ and $\mathbf{y}$ of bidegree $(d,d)$.
It is clear that if $(\mathbf{x}, \mathbf{y}) = (p_1, \ldots, p_n, q_1, \ldots, q_n) \in ([0, N^{1- \delta}]^{n} \times [0, N^{\delta}]^n) \cap \wp^{2n}$ is a prime solution to the equation
$g(\mathbf{x}; \mathbf{y}) = 0$, then $(p_1q_1, \ldots, p_n q_n) \in [0, N]^{n}$ is a semiprime solution to the equation (\ref{eqn main 2}).
%with every coordinate a semiprime.
Therefore, by taking into account possible repetitions we have
\begin{equation}
\label{main ineq 1}
\mathcal{N}_2(f; N; N^{1 - \delta}, N^{\delta} ) \geq \frac{1}{2^n} \ \mathcal{N}_{\wp}( g ; N^{1 - \delta}, N^{\delta} ).
\end{equation}
By Theorem \ref{mainineqthm}, we have
\begin{equation}
\label{final ineq 2}
\min \{ \textnormal{codim } V_{{G}, 1}^*, \textnormal{codim } V_{{G}, 2}^*  \} \geq   \frac{\textnormal{codim } V_{{F}}^*}{2} > 4^d \max \left\{  4 (2d-1), \frac{d}{\delta}  \right\}.
\end{equation}
It follows
%from (\ref{final ineq 1}) and (\ref{final ineq 2})
that the bihomogeneous form $G$ satisfies (\ref{assmpn on F}) with $d_1 = d_2 = d$, $P_1 = N^{1 - \delta}$, $P_2 = N^{\delta}$, $R=1$, and
$\mathfrak{b} = \frac{1 - \delta}{\delta}$. Therefore, Theorem \ref{thm main bihmg} gives us
\begin{equation}
\label{eqn asymp}
\mathcal{N}_{\wp}( g ; N^{1 - \delta}, N^{\delta} ) =  \sigma_g N^{n - d} + O\left( \frac{ N^{n - d} }{(\log N)^{c}} \right)
\end{equation}
for some $c>0$.

We now prove that $\sigma_g$ in (\ref{eqn asymp}) is in fact positive.
Suppose the equation (\ref{local real})
has a non-singular real solution $(\xi_1, \ldots, \xi_n) \in (0,1)^{n}$.
Then it can be verified that
$(\xi_1, \ldots, \xi_n,$ $1/2, \ldots, 1/2) \in (0,1)^{2n}$
is a non-singular real solution to
the equation $G(\mathbf{x}; \mathbf{y}) = 0.$
Similarly if the equation (\ref{eqn main 2})
has a non-singular solution $(\xi_1, \ldots, \xi_n) \in ( \mathbb{Z}_p^{\times})^n$, then
$(\xi_1, \ldots,$ $\xi_n, 1, \ldots, 1) \in ( \mathbb{Z}_p^{\times})^{2n}$
is a non-singular solution in $( \mathbb{Z}_p^{\times})^{2n}$ to the equation
$g(\mathbf{x}; \mathbf{y}) = 0$.
Thus it follows from Theorem \ref{thm main bihmg} that $\sigma_g > 0$.
Therefore, we obtain from (\ref{main ineq 1}) and  (\ref{eqn asymp}) that
$\mathcal{N}_2(f; N; N^{1 - \delta}, N^{\delta})
%\geq \mathcal{N}_{\wp}(\mathbf{g}; N^{1 - \delta}, N^{\delta} ) =  \sigma N^{n - dR} + O\left( \frac{ N^{n - dR} }{(\log N)^{c}} \right)
\gg N^{n - d}.$
%where $\sigma > 0$ and $c>0$.
\end{proof}

%\section{Introduction}
% The body of your paper goes here~\cite{cilleruelo}.

%\newpage %% AUTHOR: please comment out this line.  It serves only
%%   to demonstrate both types of header line in daj-template.pdf

%\section{Expansion estimates}

% More of the body of your paper goes here~\cite{bergelson-johnson-moreira}.

%%% AUTHOR: optional appendix here
%\appendix %% you may comment this out if no Appendix
%\section*{Appendix}
%\section{Improving the constants}
%Material is placed here as needed.

%%% AUTHOR: optional acknowledgments here
\section*{Acknowledgments} %%  you may comment this out if no Ackno
The author would like to thank Tim Browning, Brian Cook, Natalia Garcia-Fritz, Damaris Schindler, and Trevor Wooley for many helpful discussions, and
the anonymous referees for their useful comments.
A large portion of this work was completed while the author was attending
the Thematic Program on Unlikely Intersections, Heights, and Efficient Congruencing
at the Fields Institute, and he would like to thank the Fields Institute for their support and for providing an excellent environment to work on this paper.
He would also like to thank M. Ram Murty and the Department of Mathematics and Statistics at Queen's University,
and EPSRC grant \texttt{EP/P026710/1} for their support while completing this work.

%%% AUTHOR:
%%% Bibliography goes here. Note that the arXiv cannot process bibtex
%%% or biber bibliographies.  Example of acceptable bibliograpy format:
\bibliographystyle{amsplain}

\begin{thebibliography}{99}


\bibitem{B}  B. J. Birch, \textit{Forms in many variables}. Proc. Roy. Soc. Ser. A 265 1961/1962, 245--263.

\bibitem{BGS} J. Bourgain, A. Gamburd and P. Sarnak, \textit{Affine linear sieve, expanders, and sum-product}. Invent. Math. 179 (2010), no. 3, 559--644.

\bibitem{C} J. R. Chen, \textit{On the representation of a larger even integer as the sum of a prime and the product of at most two primes}.
 Sci. Sinica 16 (1973), 157--176.

%\bibitem{BHB} T.D. Browning, and D.R. Heath-Brown, \textit{Froms in many variables and differing degrees}.
%J. Eur. Math. Soc., to appear.

%\bibitem{BP} T.D. Browning, and S.M. Prendiville, \textit{Improvements in Birch's theorem on forms in many variables}.
%J. Reine Angew. Math., to appear.

%\bibitem{BDLW} J. Br\"{u}dern, R. Dietmann, J. Liu and T. D. Wooley,
%\textit{A Birch-Goldbach theorem}. Arch. Math. (Basel) 94 (2010), no. 1, 53--58.

\bibitem{CM} B. Cook and {\'A}. Magyar, \textit{Diophantine equations in the primes}. Invent. Math. {198} (2014), 701--737.

%\bibitem{C} S. Chow, \textit{Roth-Waring-Goldbach}. arXiv:1602.04012.

%\bibitem{D} H. Davenport, \textit{Analytic methods for Diophantine equations and Diopantine inequalities}. Second edition.
%Cambridge University Press, Cambridge, 2005.

%\bibitem{DRS} W. Duke, Z. Rudnick and P. Sarnak, \textit{Density of integer points on affine homogeneous varieties}. Duke Math. J. 71 (1993), no. 1, 143--179.

%\bibitem{GPY}  D. A. Goldston,  J. Pintz and  C. Y. Y{\i}ld{\i}r{\i}m, \textit{Primes in tuples. I}.  Ann. of Math. (2) 170 (2009), no. 2, 819--862.

%\bibitem{GS}  A. S. Golsefidy and P. Sarnak, \textit{The affine sieve}, J. Amer. Math. Soc. 26 (2013), no. 4, 1085–1105.

\bibitem{GT0} B. Green and  T. Tao, \textit{The primes contain arbitrarily long arithmetic progressions}. Ann. of Math. (2) 167 (2008), no. 2, 481--547.

\bibitem{GT1} B. Green and  T. Tao, \textit{Linear equations in primes}.  Ann. of Math. (2) 171 (2010), no. 3, 1753--1850.

\bibitem{GT2} B. Green and  T. Tao, \textit{The M\"obius function is asymptotically orthogonal to nilsequences}.  Ann. of Math. 175 (2012), 541--566.

\bibitem{GTZ} B. Green,  T. Tao and T. Ziegler, \textit{An inverse theorem for the Gowers $U^{s+1}[N]$ norm}.  Ann. of Math. 176 (2012), no. 2, 1231--1372.


%\bibitem{H1} H. A. Helfgott, \textit{Major arcs for Goldbach's problem}. arXiv:1305.2897.

%\bibitem{H2} H. A. Helfgott, \textit{Minor arcs for Goldbach's problem}. arXiv:1205.5252.

\bibitem{H}  L. K. Hua, \textit{Additive theory of prime numbers}. Translations of Mathematical Monographs, Vol. 13 American Mathematical Society,
Providence, RI (1965).

\bibitem{KT} A. V. Kumchev and I. D. Tolev, \textit{An invitation to additive prime number theory}. Serdica Math. J. 31 (2005), no. 1-2, 1--74.

%\bibitem{KW} A. V. Kumchev and T. D. Wooley, \textit{On the Waring-Goldbach problem for eighth and higher powers}.
%J. London Math. Soc. (2) 93 (2016), no. 3, 811--824.

%\bibitem{L} J. Liu, \textit{Integral points on quadrics with prime coordinates}.
%Monatsh. Math. 164 (2011), no. 4, 439--465.

\bibitem{LS} J. Liu and P. Sarnak, \textit{Integral points on quadrics in three variables whose coordinates have few prime factors}. Israel J. of math. 178 (2010), 393--426.

\bibitem{MT} {\'A}. Magyar and T. Titichetrakun, \textit{Almost prime solutions of diophantine systems of high rank}.
Int. J. Number Theory 13 (2017), no. 6, 1491--1514.

%\bibitem{L1} Z. Liu, \textit{Small Prime Solutions to Cubic Diophantine Equations}, Canad. Math. Bull. 56(2013), 785-794.

%\bibitem{M}  J. Maynard, \textit{Small gaps between primes}. Ann. of Math. (2) 181 (2015), no. 1, 383--413.

\bibitem{DS} D. Schindler, \textit{Bihomogeneous forms in many variables}. J. Th\'{e}orie Nombres Bordeaux 26 (2014), 483--506.

\bibitem{SS} D. Schindler and E. Sofos, \textit{Sarnak's saturation problem for complete intersections}. Mathematika 65 (2019), no. 1, 1--56.

%\bibitem{S} W. M. Schmidt, \textit{The density of integer points on homogeneous varieties}.
%Acta Math. {154} (1985), no. 3-4, 243--296.

%\bibitem{V} I. M. Vinogradov. \textit{Representation of an odd number as
%a sum of three primes}. Dokl. Akad. Nauk. SSR, 15:291--294, 1937

%\bibitem{XY} S. Y. Xiao and S. Yamagishi, \textit{Zeroes of polynomials in many variables with prime inputs}. arXiv:1512.01258.

\bibitem{VW} R. C. Vaughan and T. D. Wooley, \textit{Waring's problem: a survey}.
Number theory for the millennium III, 301--340, A. K. Peters, Natick, MA, 2002.


\bibitem{Y} S. Yamagishi, \textit{Prime solutions to polynomial equations in many variables and differing degrees}. Forum Math. Sigma 6 (2018), e19, 89 pp.

%\bibitem{Z}  Y. Zhang, \textit{Bounded gaps between primes}. Ann. of Math. (2) 179 (2014), no. 3, 1121--1174.

\bibitem{Zh} L. Zhao, \textit{The quadratic form in nine prime variables}. Nagoya Math. J., 223 (1) (2016), 21--65.

\end{thebibliography}

%% AUTHOR: You can generate such a bibliography from a .bib file by
%% running pdflatex/bibtex/pdflatex/pdflatex and then pasting the .bbl file
%% between \begin{thebibliography} and \end{bibliography}

%%% AUTHOR: Include a short description of each author following the
%%% structure below. Use the same short tags used previously.
%%% Use \imageat{} and \imagedot{} instead of "@" and "." in
%%% email addresses-this replaces the symbols with graphics to avoid
%%% e-mail address harvesting from the .pdf file
\begin{dajauthors}
\begin{authorinfo}[sy]
  Shuntaro Yamagishi\\
  Mathematisch Instituut\\
  Universiteit Utrecht\\
  Utrecht, Nederland\\
  s\imagedot{}yamagishi\imageat{}uu\imagedot{}nl %\\
  %\url{https://www.cs.elte.hu/erdos}
\end{authorinfo}
\end{dajauthors}

\end{document}